\documentclass[a4paper]{amsart}

\usepackage{amsmath,amsthm,amssymb,esint}
\usepackage[colorlinks=true,urlcolor=blue,citecolor=red,linkcolor=blue,linktocpage,pdfpagelabels,bookmarksnumbered,bookmarksopen]{hyperref}
\usepackage{enumerate,paralist}
\usepackage[pdftex]{graphicx}
	\graphicspath{{./}{figures/}}
\usepackage{tikz}
	\usetikzlibrary{calc,arrows,positioning}

\newcommand{\abs}[1]{\left\vert{#1}\right\vert}
\newcommand{\ave}[1]{\left\langle{#1}\right\rangle}
\newcommand{\cB}{\mathcal{B}}
\newcommand{\iunit}{\mathbf{i}}
\newcommand{\cK}{\mathcal{K}}
\newcommand{\Lipone}{\textup{Lip}_1(\R^d)}
\newcommand{\N}{\mathbb{N}}
\newcommand{\R}{\mathbb{R}}
\newcommand{\Z}{\mathbb{Z}}
\newcommand{\MNone}{\mathcal{M}^N_1(\R^d)}
\newcommand{\vd}{v_\textup{d}}

\let\div\relax\DeclareMathOperator{\div}{div}
\DeclareMathOperator{\Img}{\Im\textup{m}}
\DeclareMathOperator{\Lip}{Lip}
\DeclareMathOperator{\Real}{\Re\textup{e}}
\DeclareMathOperator{\supp}{supp}

\theoremstyle{plain}
\newtheorem{assumption}{Assumption}[section]
\newtheorem{lemma}[assumption]{Lemma}
\newtheorem{proposition}[assumption]{Proposition}

\newtheorem{theorem}[assumption]{Theorem}

\theoremstyle{definition}
\newtheorem{example}[assumption]{Example}

\title[Comparing microscopic and macroscopic crowd models]{Comparing first order microscopic and macroscopic crowd models for an increasing number of massive agents}
\author{	Alessandro Corbetta}
\address{Department of Structural and Geotechnical Engineering, Politecnico di Torino, Corso Duca degli Abruzzi 24, 10129 Torino, Italy \newline\indent
Department of Applied Physics and Centre for Analysis, Scientific computing and Applications, Department of Mathematics and Computer Science, Eindhoven University of Technology, P.O. Box 513, 5600 MB Eindhoven, The Netherlands}
\email{a.corbetta@tue.nl}
\author{Andrea Tosin}
\address{Istituto per le Applicazioni del Calcolo ``M. Picone'', Consiglio Nazionale delle Ricerche, Via dei Taurini 19, 00185 Roma, Italy}
\email{a.tosin@iac.cnr.it}
\thanks{The first author was funded by a ``Lagrange'' Ph.D. scholarship granted by the CRT Foundation, Torino, Italy and by Eindhoven University of Technology, Eindhoven, The Netherlands.}

\begin{document}



\begin{abstract}
In this paper a comparison between first order microscopic and macroscopic differential models of crowd dynamics is established for an increasing number $N$ of pedestrians. The novelty is the fact of considering massive agents, namely particles whose individual mass does not become infinitesimal when $N$ grows. This implies that the total mass of the system is not constant but grows with $N$. The main result is that the two types of models approach one another in the limit $N\to\infty$, provided the strength and/or the domain of pedestrian interactions are properly modulated by $N$ at either scale. This is consistent with the idea that pedestrians may adapt their interpersonal attitudes according to the overall level of congestion.
\end{abstract}

\maketitle

\section{Introduction}     
Pedestrians walking in crowds exhibit rich and complex dynamics, which in the last years generated problems of great interest for different scientific communities including, for instance, applied mathematicians, physicist, and engineers (see~\cite[Chapter 4]{cristiani2014BOOK} and~\cite{duives2013TRC,venuti2009PLR} for recent surveys). This led to the derivation of numerous mathematical models providing qualitative and possibly also quantitative descriptions of the system,~\cite{corbetta2015MBE,johansson2007ACS,zanlungo2012PONE}.

When deducing a mathematical model for pedestrian dynamics different observation scales can be considered. Two extensively used options are the microscopic and the macroscopic scales. Microscopic models describe the time evolution of the position of each single pedestrian, addressed as a discrete particle~\cite{fehrenbach2014PREPRINT,helbing1995PRE,hoogendoorn2003OCAM,maury2007ESAIM}. Conversely, macroscopic models deal with a spatially averaged representation of the pedestrian distribution, which is treated as a continuum in terms of the pedestrian density~\cite{bruno2009JSV,colombo2012M3AS,hughes2002TRB,piccoli2011ARMA,twarogowska2014AMM}. Furthermore, crowds have been also represented at the mesoscopic scale~\cite{agnelli2015M3AS,degond2013JSP,degond2013NHM} or via discrete systems such as Cellular Automata~\cite{burstedde2001PA,kirchner2002PA}.

Different observation scales serve different purposes: the microscopic scale is more informative when considering very localized dynamics, in which the action of single individuals is relevant; conversely, the macroscopic scale is appropriate when insights into the ensemble (collective) dynamics are required or  when high densities are considered. In addition to this, spatially discrete and continuous scales may provide a dual representation of a crowd useful to formalize aspects such as pedestrian perception and the interplay between individualities and collectivity~\cite{bruno2013PREPRINT,cristiani2011MMS,cristiani2014BOOK}. Selecting the most adequate representation may present difficulties, because different outcomes at different scales are likely to be observed. Nevertheless, independently of the scale, models are often deduced out of common phenomenological assumptions, hence they are expected to reproduce analogous phenomena. The question then arises when and how they are comparable to each other.

\begin{figure}[!t]
\includegraphics[width=\textwidth]{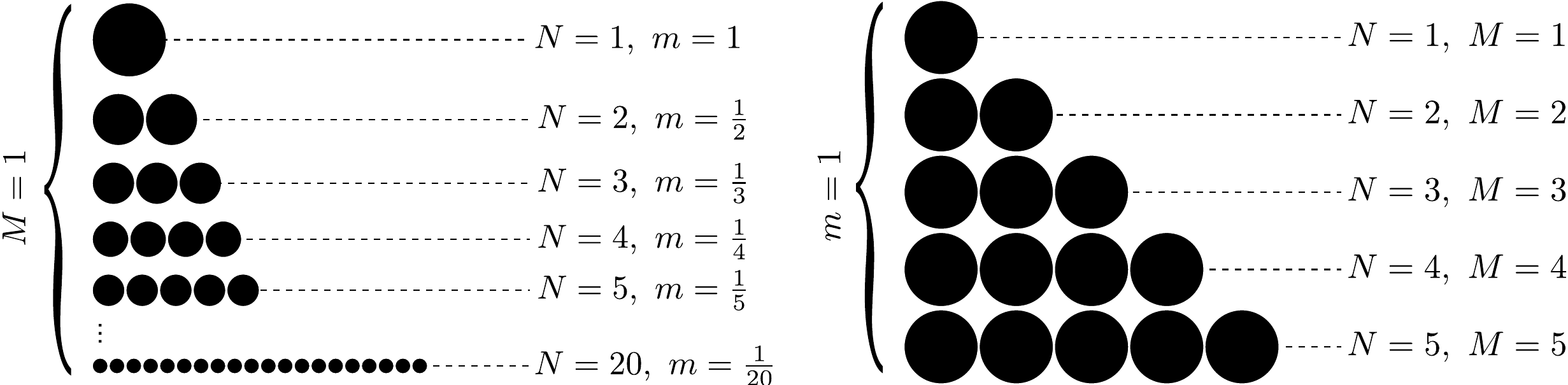}
\caption{Left: Classical mean-field point of view, in which the total mass $M$ of the system is constant while the mass $m$ of each particle becomes infinitesimal as $N$ increases. Right: The point of view pursued in this paper, in which $m$ is constant (\emph{massive particles}) while $M$ grows when $N$ increases.}
\label{fig:meanfield_vs_massive}
\end{figure}

These arguments provide the motivation for this paper, in which a comparison of microscopic and macroscopic crowd models is carried out for a growing number $N$ of pedestrians. It is well known that the statistical behavior of microscopic systems of interacting particles can be described, for $N\to\infty$, by means of a Vlasov-type kinetic equation derived in the \emph{mean field limit} under the assumption that the strength of pairwise interactions is scaled as $N^{-1}$ (\emph{weak coupling scaling}), see~\cite{carrillo2009KRM,carrillo2010MSSET} and references therein. If the total mass of the system is $M=mN$, where $m$ is the mass of each particle, this corresponds to assuming that particles generate an interaction potential in space proportional to their mass $m=M/N$ (like, e.g, in gravitational interactions). The mean field limit requires the assumption of a constant total mass of the system, say $M=1$, which implies that the mass of each particle becomes infinitesimal as $N$ grows (cf. Fig.~\ref{fig:meanfield_vs_massive} left). On the contrary, considering continuous models \emph{per se}, parallelly to discrete ones, allows one to keep the mass $m$ of each individual constant, say $m=1$, thus $M=N$ holds  (cf. Fig.~\ref{fig:meanfield_vs_massive} right). In this perspective, a comparison with discrete models based on the role of $N$ acquires a renewed interest.

This point of view is being introduced also in the context of other systems of interacting particles, such as e.g., vehicular traffic. Quoting from the conclusions of the lecture~\cite{seibold2015IPAM}:
\begin{quote}
Real traffic is microscopic. Ideally, accurate macroscopic models should not focus on the limit $N\to\infty$, but represent the solution with the true number of vehicles $N$.
\end{quote}
Pedestrian crowds are microscopic as well, hence  macroscopic crowd models should be built consistently with the phenomenology of a \emph{finite number} of microscopic \emph{massive} pedestrians.
Of course, we cannot expect the microscopic and macroscopic solutions to be the same for all numbers of pedestrians, however we can ascertain if the two types of models are actually ``the same model'' at least in some asymptotic regime. In this sense we  address the limit $N\to\infty$.

The two types of models which will be considered throughout the paper assume  first order position-dependent pedestrian dynamics, given via the walking velocity.
 Specifically, in the microscopic case, let $X^1_t,\,X^2_t,\,\dots,\,X^N_t\in\R^d$ be the positions of $N$ pedestrians at a time $t$. Their evolution  satisfies
\begin{equation}
	\dot{X}_t^i=\vd(X_t^i)-\sum_{j=1}^NK(X_t^j-X_t^j), \qquad i=1,\,\dots,\,N.
	\label{eq:micro-proto}
\end{equation}
Conversely, in the macroscopic case, let $\rho_t(x)$ be the density of pedestrians in the point $x\in\R^d$ at time $t$, such that $\int_{\R^d}\rho_t(x)\,dx=N$ for all $t\geq 0$. In some analogy with~\eqref{eq:micro-proto}, its evolution is  given by the conservation law
\begin{equation}
	\partial_t\rho_t+\div\left(\rho_t\left(\vd-\int_{\R^d}K(y-\cdot)\rho_t(y)\,dy\right)\right)=0.
	\label{eq:macro-proto}
\end{equation}
In both cases, pedestrian velocity is modeled as a sum of two terms: a \emph{desired velocity} $\vd$, which walkers would keep in the absence of others, plus \emph{repulsive} (whence the minus sign) \emph{interactions}, which perturb $\vd$ for collision avoidance purposes. By assumption, interactions depend on the relative distance between pairs of interacting pedestrians via an \emph{interaction kernel} $K$. 

Models~\eqref{eq:micro-proto} and~\eqref{eq:macro-proto} can  be recognized as particular instances of a scale-agnostic measure-valued conservation law. This abstract conservation law will play the role of a pivot in the comparison performed here.
%
%
The comparison and the paper are organized as follows: in Section~\ref{sec:mod_framework} we introduce and briefly discuss the scale-agnostic modeling framework. In Section~\ref{sec:stationary} we give a first comparison result of discrete and continuous dynamics in the one-dimensional stationary case, which leads to the computation of the so-called \emph{speed diagrams}. The main result is that the asymptotic pedestrian speeds predicted by models~\eqref{eq:micro-proto},~\eqref{eq:macro-proto} are not the same and that they cannot even be expected to match one another for large numbers of pedestrians. In Section~\ref{sec:non-stationary} we then give a second more complete comparison result in a general $d$-dimensional time-evolutionary setting. We consider sequences of pairs of discrete and continuous models of the form~\eqref{eq:micro-proto},~\eqref{eq:macro-proto} indexed by the total number $N$ of pedestrians and, as main contributions, we establish:
\begin{inparaenum}[(i)]
\item for fixed $N$, a stability estimate relating the distance between those pairs of models at a generic instant $t>0$ to the one at the initial time $t=0$;
\item for fixed $t$, a family of scalings of the interactions, comprising the afore-mentioned mean field as a particular case, which control the amplification of such a distance when $N$ grows;
\item a procedure to construct discrete and continuous initial configurations of the crowd giving rise to mutually convergent sequences of discrete and continuous models at all times when $N$ grows.
\end{inparaenum}
Finally, in Section~\ref{sec:discussion} we discuss the implications of the obtained results on the modeling of crowd dynamics and other particle systems in a multiscale perspective.

\section{A scale-agnostic modeling framework}
\label{sec:mod_framework}
Models~\eqref{eq:micro-proto} and~\eqref{eq:macro-proto} are particular cases of the measure-valued conservation law in $\R^d$
\begin{equation}
	\partial_t\mu_t+\div(\mu_tv[\mu_t])=0,
	\label{eq:cont-gen}
\end{equation} 
where $t$ is the time variable, $\mu_t$ is a time-dependent spatial measure of the crowding and $v[\mu_t]$ is a measure-dependent velocity field to be prescribed (see below). More specifically, $\mu_t$ is a positive locally finite Radon measure defined on $\cB(\R^d)$, the Borel $\sigma$-algebra on $\R^d$, which satisfies
\begin{equation}
	\mu_t(\R^d)=N, \qquad \forall\,t\geq 0,
	\label{eq:number-of-ped}
\end{equation}
where, we recall, $N\in\N$ is the (conserved) number of pedestrians. The cases $d=1,\,2$ are commonly considered to model, respectively, scenarios in which pedestrians are aligned e.g., along a walkway ($d=1$) or can walk in a given planar area ($d=2$). Equation~\eqref{eq:cont-gen} has to be understood in the proper weak formulation:
\begin{equation}
	\int_{\R^d}\phi\,d\mu_t=\int_{\R^d}\phi\,d\mu_0+\int_0^t\int_{\R^d}\nabla\phi\cdot v[\mu_s]\,d\mu_s\,ds, \qquad \forall\,\phi\in C^\infty_c(\R^d),
	\label{eq:weak-form}
\end{equation}
$C^\infty_c(\R^d)$ being the space of infinitely smooth and compactly supported test functions $\phi:\R^d\to\R$. From~\eqref{eq:weak-form} it can be formally checked that, given an initial measure $\mu_0$ (initial configuration of the crowd), at time $t>0$ it results
\begin{equation}
	\mu_t=\gamma_t\#\mu_0 \quad \textup{i.e.} \quad \mu_t(E)=\mu_0(\gamma_t^{-1}(E)) \quad \forall\,E\in\cB(\R^d),
	\label{eq:pushfwd}
\end{equation}
where $\#$ is the \textit{push forward} operator (see e.g.,~\cite{ambrosio2008BOOK}) and $\gamma_t$ is the \emph{flow map} defined as
\begin{equation}
	\gamma_t(x)=x+\int_0^t v[\mu_s](\gamma_s(x))\,ds.
	\label{eq:def-flow-map}
\end{equation}
It is worth stressing that, under proper regularity conditions on the velocity $v$, all solutions $t\mapsto\mu_t$ of the Cauchy problem associated to~\eqref{eq:cont-gen} are continuous functions of time~\cite{piccoli2013AAM,tosin2011NHM} and admit the representation~\eqref{eq:pushfwd}-\eqref{eq:def-flow-map}~\cite{ambrosio2008BOOK}.

Such a measure-based framework features an intrinsic generality, indeed it can describe a discrete crowd distribution when $\mu_t$ is an atomic measure:
\begin{equation}
	\mu_t=\epsilon_t:=\sum_{i=1}^N\delta_{X^i_t},
	\label{eq:construction-of-meas-epsilon}
\end{equation}
where $\{X^i_t\}_{i=1}^{N}\subset\R^d$ are the positions of pedestrians at time $t$, or a continuous crowd distribution when $\mu_t$ is an absolutely continuous measure with respect to the $d$-dimensional Lebesgue measure:
\begin{equation}
	\mu_t=\rho_t
	\label{eq:construction-of-meas-rho}
\end{equation}
with $\rho_t(\R^d)=\int_{\R^d}d\rho_t=N$ for all $t\geq 0$. In this latter case, by Radon-Nykodim theorem, $\mu_t$   admits a density, i.e., the crowd density. For ease of notation, we will systematically confuse the measure $\rho_t$ with its density and write $d\rho_t(x)$ or $\rho_t(x)\,dx$ interchangeably. Moreover, when required by the context, we will write explicitly the number $N$ of pedestrians as superscript of the measures (e.g., $\mu_t^N$, $\epsilon^N_t$, and $\rho^N_t$).

Once plugged into~\eqref{eq:weak-form}, the measures~\eqref{eq:construction-of-meas-epsilon},~\eqref{eq:construction-of-meas-rho} produce models~\eqref{eq:micro-proto},~\eqref{eq:macro-proto}, respectively, if the following velocity is used:
\begin{equation}
	v[\mu_t](x)=\vd(x)-\int_{\R^d}K(y-x)\,d\mu_t(y).
	\label{eq:velocity-transport-eq}
\end{equation}
The interaction kernel $K:\R^d\to\R^d$ represents pairwise interactions occurring among walking pedestrians, which, in normal conditions (i.e., no panic), tend to be of mutual avoidance and finalized at maintaining a certain comfort distance. As a consequence, they are supposed to be repulsive-like, whence $-K(z)\cdot z\leq 0$ for all $z\in\R^d$. Moreover, they are known to happen within a bounded region in space, the so-called \emph{sensory region}~\cite{fruin1971BOOK}. Therefore $K$ has compact support, which, for a pedestrian in position $x$, we denote by
\[ S_R(x):=\supp{K(\cdot-x)}\subseteq B_R(x), \]
$B_R(x)$ being the ball centered in $x$ with radius $R$. Therefore, $R$ is the maximum distance from $x$ at which interactions are effective. If the sensory region is not isotropic (as it is the case for pedestrians, who interact preferentially with people ahead), its orientation is expected to depend on the pedestrian gaze direction~\cite{fruin1971BOOK}. Nonetheless, in the following, we assume for simplicity that $S_R(x)$ is just a rigid translation of a prototypical region $S_R(0)=\supp{K}\subseteq B_R(0)$. Extending the proposed setting to models featuring fully orientation dependent sensory regions is mainly a technical issue, for which the reader can refer e.g., to~\cite{evers2014PREPRINT,tosin2011NHM}.

According to the arguments set forth, \eqref{eq:cont-gen} allows for a formal qualitative correspondence between the two modeling scales, nonetheless no quantitative correspondence is established  between the actual dynamics. The analysis of quantitative correspondences will be the subject of the next sections.

\section{The one-dimensional stationary case: speed diagrams}
\label{sec:stationary}
In this section we study and compare the stationary behavior of one-dimensional ($d=1$) microscopic and macroscopic homogeneous pedestrian distributions satisfying~\eqref{eq:cont-gen} with velocity~\eqref{eq:velocity-transport-eq}. Homogeneous conditions, yet to be properly defined at the two considered scales, represent dynamic equilibrium conditions possibly reached asymptotically, after a transient. In homogeneous conditions, the speed of pedestrians is expected to be a constant, depending exclusively on the number of pedestrians and on the length, say $L>0$, of the one-dimensional domain.    
 
The evaluation of the pedestrian speed in homogeneous crowding conditions is an established experimental practice which leads to the so-called \textit{speed diagrams}, i.e., synthetic quantitative relations between the density of pedestrians and their average speed~\cite{seyfried2008CA}. Usually, such diagrams feature a decreasing trend for increasing values of the density, and are defined up to a characteristic value (stopping density) at which the measured speed is zero. In the following, speed diagrams are studied as a function of the number of pedestrians $N$ (for an analogous experimental case cf.~\cite{corbetta2014TRP}), in the microscopic and macroscopic cases. In this context, $N$ is retained as the common element between the two descriptions, as it remains well-defined independently of the observation scale. It is finally worth pointing out that, although speed diagrams have been often used in mathematical models as closure relations (especially for macroscopic models, cf.~\cite{bruno2011AMM,colombo2009NARWA}), in this case they are a genuine output of the considered interaction rules expressed by the integral in~\eqref{eq:velocity-transport-eq}.  
     
\subsection{Modeling hypotheses}
\label{sec:hyp_stat}
For the sake of simplicity, we consider the one-di\-men\-sio\-nal problem on a periodic domain $[0,\,L)$. In order to model homogeneous pedestrian distributions, in the discrete case~\eqref{eq:micro-proto} we consider an equispaced lattice solution translating with a certain constant speed $w$ (to be determined), i.e.,
\begin{equation}
	\epsilon_t(x)=\bar{\epsilon}(x-wt),
	\label{eq:epsilonN-timet}
\end{equation}
where $\bar{\epsilon}$ has the form~\eqref{eq:construction-of-meas-epsilon} with atom locations such that
\begin{equation}
	\abs{\bar{X}^j-\bar{X}^i}=(j-i)\frac{L}{N}, \qquad i,\,j=1,\,\dots,\,N,
	\label{eq:equispaced-lattice}
\end{equation}
for $\bar{X}^1<\bar{X}^2<\dots<\bar{X}^N$ (the ordering is modulo $L$). Thus the atoms of $\epsilon_t$ are $X^i_t=\bar{X}^i+wt$ for $i=1,\,\dots,\,N$. In the continuous case~\eqref{eq:macro-proto}, we consider instead a constant density, i.e.,
\begin{equation}
	\rho_t(x)\equiv\bar{\rho}\geq 0,
	\label{eq:homog-meas}
\end{equation}
which, owing to~\eqref{eq:number-of-ped}, is given by $\bar{\rho}=N/L$.

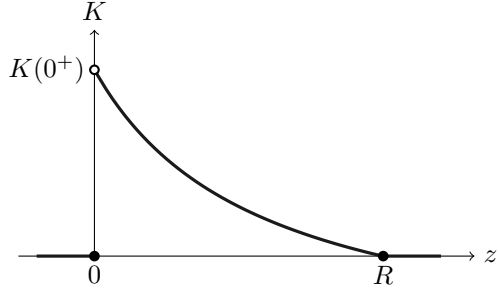
\begin{figure}[!t]
\begin{tikzpicture}[scale=2, Name/.style = {font={\bfseries}}]

     \draw[->] (-0.5,0) -- (2.5,0) node[right] {$z$};
     \draw[->] (0,0) -- (0,1.5) node[above] {$K$};

    \coordinate (origin) at (0,0);
    \coordinate (endl) at (4,0);
    \coordinate (X2) at (.8,0);
    \coordinate (X3) at (1.6,0);
    \coordinate (X4) at (2.4,0);
    \coordinate (X5) at (3.2,0);



    \newcommand*{\costA}{0.2}
    \newcommand*{\costR}{0.35}
    \newcommand*{\costXs}{2.85}

    \draw[scale=1.9,domain=0:1,smooth, variable=\x,black!90,very thick]  plot ({\x},{ .65*\costA/\costR * (\costR-\x/\costXs)/((\x/\costXs + \costA))  }); 
    \draw[scale=1.9,domain=-.2:0,smooth, variable=\x,black!90,very thick] plot({\x},{0});
    \draw[scale=1.9,domain=1:1.2,smooth, variable=\x,black!90,very thick] plot({\x},{0});

    \fill[black] (1.9,0) node [below] {$R$} circle (1pt);
    \fill[black] (0,1.9*.65) node [left] {$K(0^+)$} circle (1pt);
    \fill[white] (0,1.9*.65) circle (.6pt);

    \fill[black] (0,0) node [below] {$0$}   circle (1pt);
    
\end{tikzpicture}

\caption{Prototype of interaction kernel $K$ complying with Assumption~\ref{ass:K-properties}.}
\label{fig:K}
\end{figure}

We assume that the desired velocity is a positive constant $\vd>0$, therefore the movement is in the positive direction of the real line. Furthermore, we make the following assumptions on the interaction kernel (cf. Fig.~\ref{fig:K}):

\begin{assumption}[Properties of $K$]\hfill
\begin{enumerate}[(i)]
\item \label{ass:K-compact_supp} \emph{Compactness of the support and frontal orientation of the sensory region}. The support of $K$ is
\[ S_R(0)=[0,\,R] \]
with $0<R<L$.
\item \label{ass:K-smooth_bounded} \emph{Boundedness and regularity in $(0,\,R)$}. Pedestrian interactions vary smoothly with the mutual distance of the interacting individuals and have a finite maximum value. Specifically:
\[ K\in C^2(0,\,R), \quad K,\,K''\in L^\infty(0,\,R). \]
\item \label{ass:K-monotone} \emph{Monotonicity in $(0,\,R)$ and behavior at the endpoints}. We assume
\[ K(z)>0, \quad K'(z)<0 \quad \textup{for\ } z\in (0,\,R) \]
with moreover
\begin{align*}
	& K(0)=K(R)=0 \\
	& K(0^+):=\lim_{z\to 0^+}K(z)>0.
\end{align*}
Thus pedestrian interactions decay in the interior of the sensory region as the mutual distance increases and, moreover, pedestrians do not ``self-interact''.  This forces $K$ to be discontinuous in $z=0$.
\end{enumerate}
\label{ass:K-properties}
\end{assumption}

\subsection{Stationarity and stability of spatially homogeneous solutions}
\label{sec:stability}
Before proceeding with the comparison of asymptotic pedestrian speeds resulting from microscopic and macroscopic dynamics, we ascertain that the spatially homogeneous solutions~\eqref{eq:epsilonN-timet}-\eqref{eq:equispaced-lattice} and~\eqref{eq:homog-meas} are indeed stable and possibly attractive solutions to either~\eqref{eq:micro-proto} or~\eqref{eq:macro-proto}. This ensures that such distributions can indeed be considered as equilibrium distributions, and therefore that the evaluation of speed diagrams is well-posed.

The recent literature about discrete and continuous models of collective motions is quite rich in contributions dealing with the stability of special patterns, such as e.g., flocks, mills, and double mills, see~\cite{carrillo2009KRM,carrillo2010SIMA,carrillo2013PD,dorsogna2006PRL} and references therein. We consider here  much simpler one-dimensional configurations, however useful in this context because they reproduce mathematically the typical experimental setups in which pedestrian speed diagrams are measured, see e.g.~\cite{seyfried2010CA,zhang2014PLA}.

\begin{proposition}[Equilibrium of the microscopic model] The equispaced lattice solution~\eqref{eq:epsilonN-timet} is a stable solution to~\eqref{eq:micro-proto} for 
\begin{equation}
	w=\vd-\sum_{h=1}^{N-1}K\left(h\frac{L}{N}\right).
	\label{eq:vel-lattice}
\end{equation}
It is moreover attractive if $R>L/N$.
\label{prop:stability_micro}
\end{proposition}

Proposition~\ref{prop:stability_micro} asserts that the equispaced pedestrian distribution is always a stable (quasi-)stationary solution to the microscopic model. This is somehow in contrast to what is found in microscopic \emph{optimal-velocity} traffic models, where the so-called \emph{POMs} (``Ponies-on-a-Merry-Go-Round'') \emph{solutions} can generate instabilities (traffic jams) depending on the total number of vehicles~\cite{bando1995PRE,seidel2009SIADS}. The rationale for this difference is that, unlike the present case, in such models vehicle interactions can be both repulsive and attractive depending on the distance of the interacting pairs.

\begin{proposition}[Equilibrium of the macroscopic model]
The spatially homogeneous solution~\eqref{eq:homog-meas} is a locally stable and attractive solution to~\eqref{eq:macro-proto}.
\label{prop:stability_macro}
\end{proposition}

For the sake of completeness, we report the proofs of  Propositions~\ref{prop:stability_micro} and~\ref{prop:stability_macro} in Appendix~\ref{app:proofs}.

\subsection{Discrete and continuous speed diagrams} 
We now calculate and compare the speed diagrams corresponding to the stable stationary homogeneous solutions studied in the previous sections, i.e., the mappings $N\mapsto v[\epsilon^N_t]$ and $N\mapsto v[\bar{\rho}^N]$, respectively.

Specifically, from Proposition~\ref{prop:stability_micro} we know
\begin{equation}
	 v[\epsilon_t^N]=\vd-\sum_{h=1}^{N-1}K\left(h\frac{L}{N}\right)
	\label{eq:speed-diagr-eps}
\end{equation}
while from~\eqref{eq:velocity-transport-eq} with $d=1$, $\mu_t=\bar{\rho}^N=N/L$ and taking Assumption~\ref{ass:K-properties} into account we deduce
\begin{equation}
 	v[\bar{\rho}^N]=\vd-\frac{N}{L}\int_0^L K(z)\,dz.
	\label{eq:speed-diagr-rho}
\end{equation}

Notice that both $v[\epsilon^N_t]$ and $v[\bar{\rho}^N]$ are decreasing functions of $N$, the trend being definitely linear in the continuous case. This is consistent with typical speed diagrams for pedestrians reported in the experimental literature, see e.g.,~\cite{daamen2004PhD,polus1983JTE}.

In order to compare the two speed diagrams we introduce the quantity
\[ \Delta{v}(N):=v[\epsilon^N_t]-v[\bar{\rho}^N]. \]
Actually, since in view of Proposition~\ref{prop:stability_micro} the equilibrium speed $v[\epsilon^N_t]$ depends only on the headways $\vert X^j_t-X^i_t\vert$, which are constant in time, we can drop the dependence on $t$ by freezing pedestrians in a particular configuration, for instance the one with $\bar{X}^i=(i-1)\frac{L}{N}$. Hence we will write simply $\bar{\epsilon}^N=\sum_{i=1}^{N}\delta_{\bar{X}^i}$.

\begin{figure}[!t]
\centering
\begin{tikzpicture}[scale=2, Name/.style = {font={\bfseries}}]
    \draw  [<->,thick] (0,1.5) node (yaxis) [left] {}
        |- (4.5,0) node (xaxis) [right] {$z$};

    \coordinate (origin) at (0,0);
    \coordinate (endl) at (4,0);
    \coordinate (X2) at (.8,0);
    \coordinate (X3) at (1.6,0);
    \coordinate (X4) at (2.4,0);
    \coordinate (X5) at (3.2,0);

    \fill[gray] (origin) rectangle (.4,.5) node [below left, text=black] {$E^1_{1/2}$};
    \fill[gray!40!black] (.4,0) rectangle (1.2,.5) node [below left, text=white] {$E^2$}; 
    \fill[gray] (1.2,0) rectangle (2.0,.5) node [below left, text=black] {$E^3$};
    \fill[gray!40!black] (2.,0) rectangle (2.8,.5) node [below left, text=white] {$E^4$};
    \fill[gray] (2.8,0) rectangle (3.6,.5) node [below left, text=black] {$E^5$};
    \fill[gray!40!black] (3.6,0) rectangle (4,.5) node [below left, text=white] {$E^6_{1/2}$};

    \fill[black] (origin) node  [below] {$\bar{X}^1=0$} circle (1pt);
    \fill[black] (X2) node [below] {$\bar{X^2}$} circle (1pt);
    \fill[black] (X3) node [below] {$\bar{X^3}$} circle (1pt);
    \fill[black] (X4) node [below] {$\bar{X^4}$} circle (1pt);
    \fill[black] (X5) node [below] {$\bar{X^5}$} circle (1pt);
    \fill[black] (endl) node [below] {$L$};

    \fill[black] (endl) circle (1pt);
    \fill[white] (endl) circle (.6pt);
    \newcommand*{\costA}{0.2}
    \newcommand*{\costR}{0.35}
    \newcommand*{\costXs}{2.85}

    \draw[scale=1.9,domain=0:1,smooth, dashed,variable=\x,red,thick]  plot ({\x},{ .65*\costA/\costR * (\costR-\x/\costXs)/((\x/\costXs + \costA))  }); 

    \fill[red] (1.9,0) node [below] {$R$} circle (.9pt);
    \fill[red] (0,1.9*.65) node [left] {$K(0^+)$} circle (.9pt);
    \fill[white] (0,1.9*.65)  circle (.55pt);
\end{tikzpicture}
\caption{The considered partition of the domain $[0,\,L)$ is shown in the case $N=5$. Solid dots correspond to pedestrian positions. The dashed line portrays an example of interaction kernel $K$. }
\label{fig:assumptions-fdiagr} 
\end{figure}
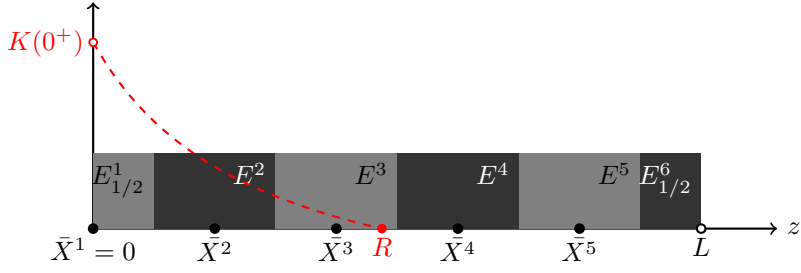

From Assumption~\ref{ass:K-properties}\eqref{ass:K-smooth_bounded} on the regularity of $K$, we can calculate $\Delta{v}(N)$ explicitly. To do that, we preliminarily introduce the following pairwise disjoint partition of the interval $[0,\,L]$ (cf. Fig.~\ref{fig:assumptions-fdiagr}):
\begin{align*}
	E_{1/2}^1 &= \left[0,\,\frac{L}{2N}\right] \\
	E^i &= \left(\left(i-\frac{3}{2}\right)\frac{L}{N},\,\left(i-\frac{1}{2}\right)\frac{L}{N}\right], \quad 2\leq i\leq N \\
	E_{1/2}^{N+1} &= \left(\left(N-\frac{1}{2}\right)\frac{L}{N},\,L\right],
\end{align*}
which is such that $\bar{X}^1\in E_{1/2}^1$, $\bar{X}^i\in E^i$ for $i=2,\,\dots,\,N$, while $E_{1/2}^{N+1}$ does not contain any of the atoms of $\bar{\epsilon}^N$. Then we have
\begin{align*}
	\Delta{v}(N) &= \int_0^L K(z)\,d(\bar{\rho}^N-\bar{\epsilon}^N)(z) \\
	& =\int_{E_{1/2}^1}K(z)\,d(\bar{\rho}^N-\bar{\epsilon}^N)(z)+\sum_{i=2}^N\int_{E^i}K(z)\,d(\bar{\rho}^N-\bar{\epsilon}^N)(z) \\
	&\phantom{=} +\int_{E_{1/2}^{N+1}}K(z)\,d\bar{\rho}^N(z)
\end{align*}
and in particular we compute:
\begin{itemize}
\item for the first integral,
\begin{align}
	\begin{aligned}[b]
		\int_{E_{1/2}^1}K(z)\,d(\bar{\rho}^N-\bar{\epsilon}^N)(z) &= \frac{N}{L}\int_{E_{1/2}^1}K(z)\,dz-K(\bar{X}^1) \\
		& =\frac{1}{2}\ave{K}_{E_{1/2}^1}-K(0)=\frac{1}{2}\ave{K}_{E_{1/2}^1}
	\end{aligned}
	\label{eq:first_int}
\end{align}
because $\bar{X}^1=0$ in the chosen configuration and moreover $K(0)=0$ (cf. Assumption~\ref{ass:K-properties}\eqref{ass:K-monotone}). We have denoted $\ave{K}_E:=\fint_E K(z)\,dz$, where $\fint$ is the integral mean;
\item for each of the integrals in the sum,
\[ \int_{E^i}K(z)\,d(\bar{\rho}^N-\bar{\epsilon}^N)(z)=\frac{N}{L}\int_{E^i}K(z)\,dz-K(\bar{X}^i)=\ave{K}_{E^i}-K(\bar{X}^i); \]
\item for the last integral,
\[ \int_{E_{1/2}^N}K(z)\,d\bar{\rho}^N(z)=\frac{N}{L}\int_{E_{1/2}^{N+1}}K(z)\,dz=\frac{1}{2}\ave{K}_{E_{1/2}^{N+1}}. \]
\end{itemize}
It follows
\begin{equation}
	\Delta{v}(N)=\frac{1}{2}\ave{K}_{E_{1/2}^1}+\sum_{i=2}^N(\ave{K}_{E^i}-K(\bar{X}^i))+\frac{1}{2}\ave{K}_{E_{1/2}^{N+1}}.
	\label{eq:deltav-evaluated}
\end{equation}

\begin{figure}[!t]
\includegraphics[width=0.8\textwidth]{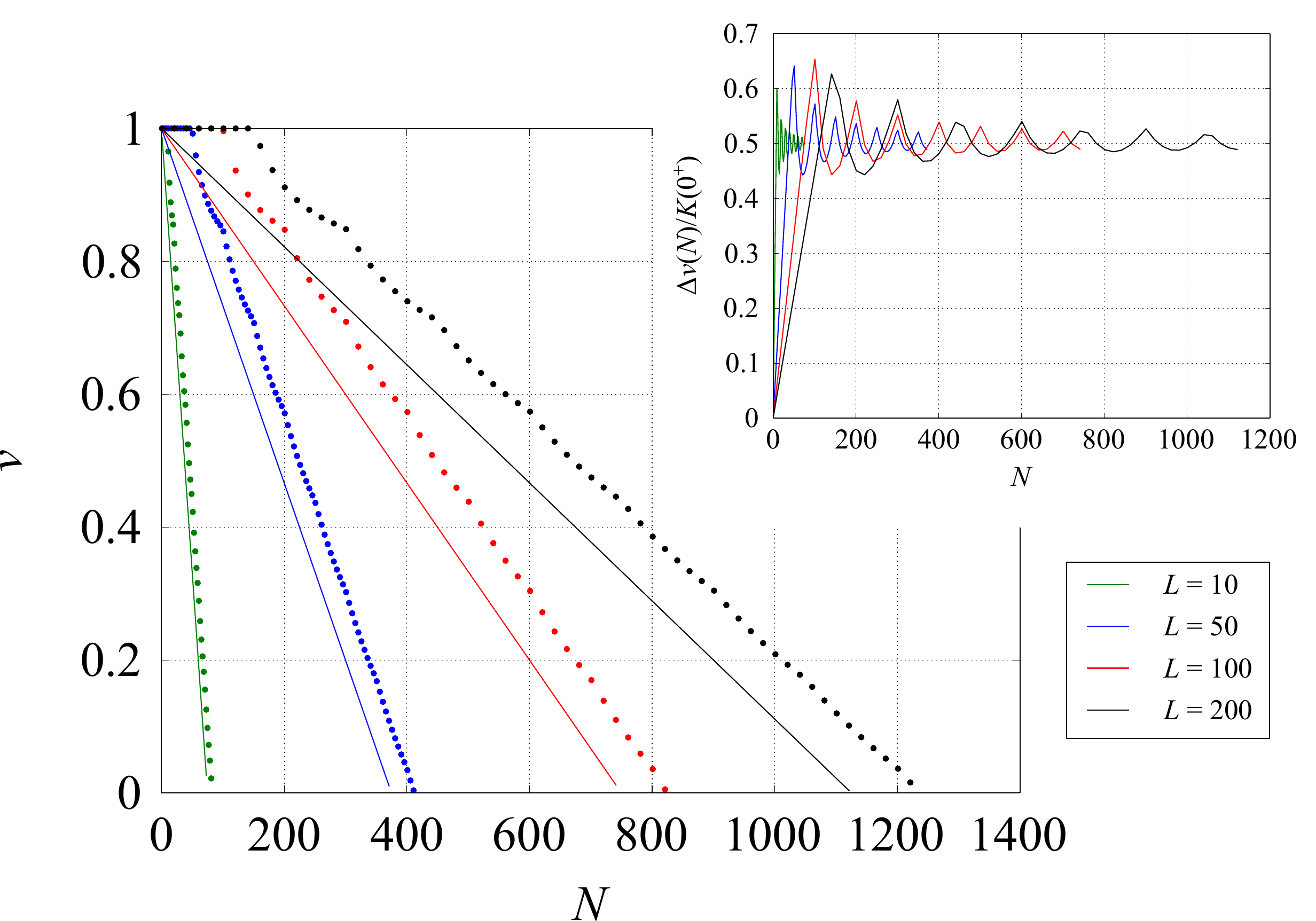}
\caption{Speed diagrams~\eqref{eq:speed-diagr-eps} (discrete model, dots) and~\eqref{eq:speed-diagr-rho} (continuous model, solid lines) obtained, for different values of $L$, with $\vd=1$, $K(z)=\frac{1}{5}(1-z^2)$. The graph in the upper-right box shows the convergence to $0.5$ of the quantity $\Delta{v}(N)/K(0^+)$ for $N\to\infty$.}
\label{fig:speed_diag-no_conv}
\end{figure}

A numerical evaluation of $\Delta{v}(N)$ is reported in Fig.~\ref{fig:speed_diag-no_conv} (up to a scaling with respect to $K(0^+)$). We observe that, when $N$ grows, the normalized curves approach the constant $\frac{1}{2}$, thus suggesting that $\Delta{v}(N)$ does not converge to $0$ for $N\to\infty$. This intuition is confirmed by the following

\begin{theorem}[Non-convergence of speed diagrams]
We have
\[ \lim_{N\to\infty}\Delta{v}(N)=\frac{1}{2}K(0^+). \]
\label{theo:convergence_stationary}
\end{theorem}
\begin{proof}
We consider one by one the terms at the right-hand side of~\eqref{eq:deltav-evaluated}.

First, the right endpoint of $E^1_{1/2}$ approaches the origin when $N$ grows, hence, owing to the mean value theorem and using the continuity of $K$ in $(0,\,R)$, cf. Assumption~\ref{ass:K-properties}\eqref{ass:K-smooth_bounded}, we find
\[ \lim_{N\to\infty}\frac{1}{2}\ave{K}_{E_{1/2}^1}=\frac{1}{2}K(0^+). \] 

Second, in view of the smoothness of $K$ in $(0,\,R)$, cf. Assumption~\ref{ass:K-properties}\eqref{ass:K-smooth_bounded}, for each term of the sum we can use a second order Taylor expansion with Lagrange remainder to discover
\begin{align*}
	\ave{K}_{E^i} &- K(\bar{X}^i) \\
	& =\fint_{E^i}\left(K(\bar{X}^i)+K'(\bar{X}^i)(z-\bar{X}^i)+\frac{1}{2}K''(\zeta^i)(z-\bar{X}^i)^2\right)\,dz-K(\bar{X}^i)
	\intertext{(for some $\zeta^i\in E^i$)}
	& =\fint_{E^i}\left(K'(\bar{X}^i)(z-\bar{X}^i)+\frac{1}{2}K''(\zeta^i)(z-\bar{X}^i)^2\right)\,dz \\
	& =\frac{1}{2}\fint_{E^i}K''(\zeta^i)(z-\bar{X}^i)^2\,dz
	\intertext{(because $z-\bar{X}^i$ is odd in $E^i$)}
	& \leq\frac{1}{2}\|K''\|_\infty\fint_{E^i}(z-\bar{X}^i)^2\,dz
		=\frac{1}{2}\|K''\|_\infty\frac{L^2}{12N^2}\xrightarrow{N\rightarrow\infty} 0
\end{align*}
which, for the arbitrariness of $i$, implies that the whole sum vanishes for $N\to\infty$.

Finally,
\[ \lim_{N\to\infty}\frac{1}{2}\ave{K}_{E_{1/2}^{N+1}}=0 \]
because of Assumption~\ref{ass:K-properties}\eqref{ass:K-compact_supp}, indeed for $N$ large it results $E_{1/2}^{N+1}\cap [0,\,R]=\emptyset$.
\end{proof}

From Theorem~\ref{theo:convergence_stationary} we conclude that the discrete system moves asymptotically at a higher speed than the continuous one because $K(0^+)>0$, as also Fig.~\ref{fig:speed_diag-no_conv} confirms. Ultimately, the discrete and continuous models~\eqref{eq:micro-proto},~\eqref{eq:macro-proto} predict different walking speeds at equilibrium, which do not match one another even in the limit of a large number of pedestrians. It is worth noticing that this fact does not actually depend on Assumption~\ref{ass:K-properties}\eqref{ass:K-monotone}, which states the absence of self-interactions ($K(0)=0$). Indeed, assuming the right continuity of $K$ in $0$, i.e., $K(0)=K(0^+)$, would still lead to a nonzero limit for $\Delta{v}(N)$: 
\[ \lim_{N\to\infty}\Delta{v}(N)=\lim_{N\to\infty}\frac{1}{2}\ave{K}_{E_{1/2}^1}-K(0)=-\frac{1}{2}K(0)<0, \]
cf.~\eqref{eq:first_int}. In this case, the discrete system is asymptotically slower than the continuous one, because discrete pedestrians are further slowed down by self-interactions (which, instead, do not affect the continuous system).

\begin{figure}[!t]
\includegraphics[width=0.8\textwidth]{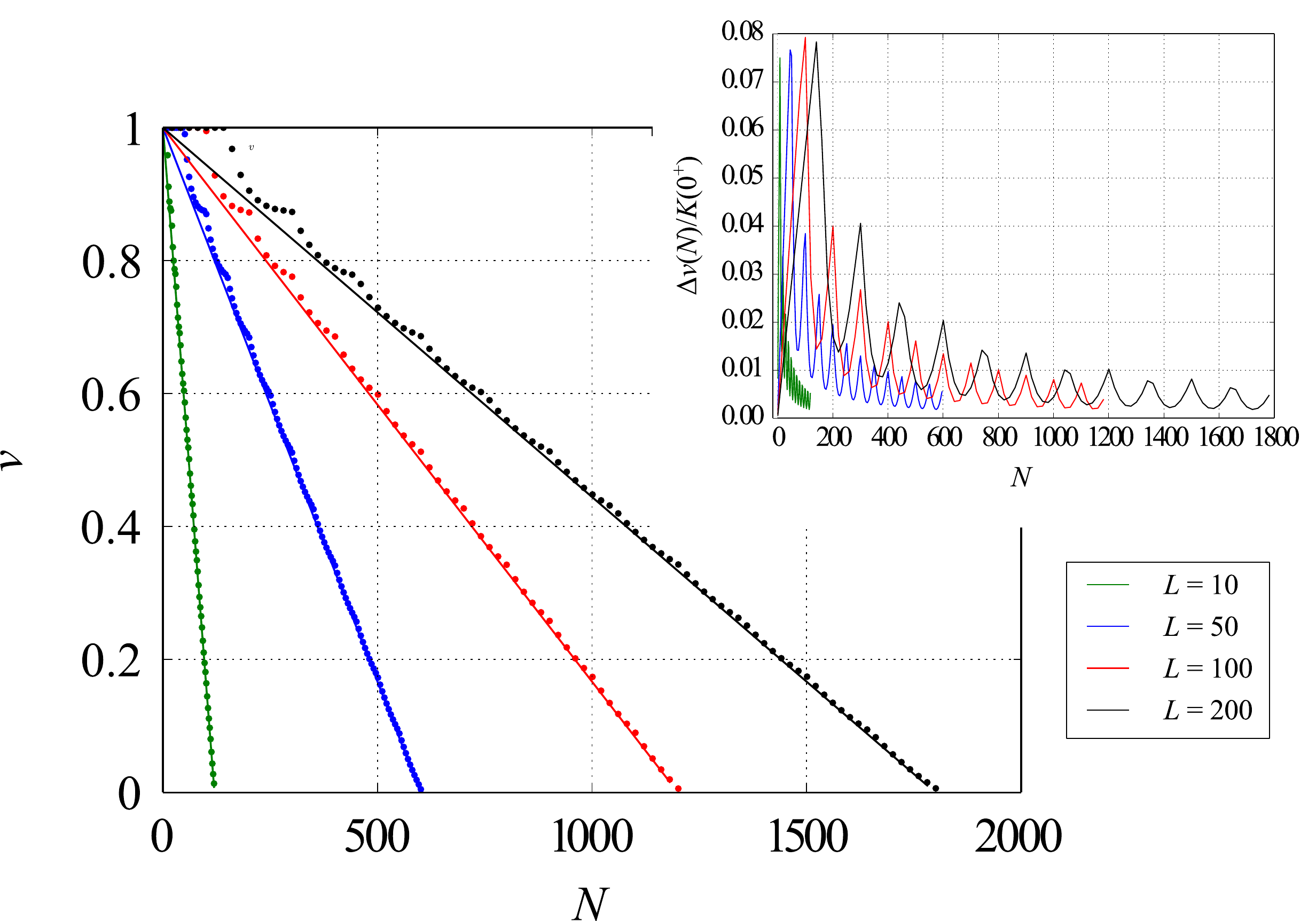}
\caption{Speed diagrams~\eqref{eq:speed-diagr-eps} (discrete model, dots) and~\eqref{eq:speed-diagr-rho} (continuous model, solid lines) obtained, for different values of $L$, with $\vd=1$, $K(z)=\frac{1}{2}z(1-z)$. The graph in the upper-right box shows the convergence to $0$ of the quantity $\Delta{v}(N)$ for $N\to\infty$.}
\label{fig:speed_diag-conv}
\end{figure}

Additionally, from Theorem~\ref{theo:convergence_stationary} we infer that speed diagrams approaching one another can be obtained if $K(0^+)=K(0)=0$, cf. Fig.~\ref{fig:speed_diag-conv}. This condition however violates the assumption that $K$ is decreasing in $(0,\,R)$, cf. Assumption~\ref{ass:K-properties}\eqref{ass:K-monotone}, which is used to prove the stability and attractiveness of the homogeneous configurations on which speed diagrams are based (cf. Propositions~\ref{prop:stability_micro},~\ref{prop:stability_macro}).

\section{General non-stationary dynamics}
\label{sec:non-stationary}
In this section we consider the Cauchy problem
\begin{equation}
	\begin{cases}
		\partial_t\mu_t+\div(\mu_tv[\mu_t])=0, & t\in (0,\,T],\ x\in\R^d \\
		\mu_0=\bar{\mu},
	\end{cases}
	\label{eq:cauchy-problem}
\end{equation}
where $v$ is given by~\eqref{eq:velocity-transport-eq}, $T>0$ is a certain final time, and $\bar{\mu}$ is a prescribed measure representing, at the proper scale, the initial distribution of the crowd. Using the results in~\cite{cristiani2014BOOK,tosin2011NHM} we can state that, under suitable assumptions encompassing those that we will recall later in Section~\ref{sec:hyp_non-stat}, problem~\eqref{eq:cauchy-problem} admits a unique measure-valued solution $\mu_\cdot\in C([0,\,T];\,\MNone)$ in the weak sense~\eqref{eq:weak-form}, $\MNone$ being the space of positive measures on $\R^d$ with total mass $N$ and finite first order moment. Moreover, such a solution preserves the structure of the initial datum: if $\bar{\mu}$ is discrete, respectively continuous, then so is $\mu_t$ for all $t\in (0,\,T]$ (in the continuous case, this is true under the further condition $\Lip(v)Te^{\Lip(v)T}<1$,~\cite{cristiani2014BOOK}).

It makes thus sense to consider sequences of discrete and continuous initial conditions $\{\bar{\epsilon}^N\}_{N=1}^{\infty},\,\{\bar{\rho}^N\}_{N=1}^{\infty}$, with $\bar{\epsilon}^N,\,\bar{\rho}^N\in\MNone\ \forall\,N\geq 1$, to which there correspond sequences of solutions at the same scales $\{\epsilon^N_\cdot\}_{N=1}^{\infty},\,\{\rho^N_\cdot\}_{N=1}^{\infty}$, with $\epsilon^N_\cdot,\,\rho^N_\cdot\in C([0,\,T];\,\MNone)\ \forall\,N\geq 1$, and to compare them ``$N$-by-$N$'' in order to determine when mutually approaching initial measures, i.e.,
\begin{equation}
	\lim_{N\to\infty}\operatorname{d}(\bar{\epsilon}^N,\,\bar{\rho}^N)=0
	\label{eq:approach_init}
\end{equation}
for some metric $\operatorname{d}$, generate mutually approaching solutions, i.e.,
\begin{equation}
	\lim_{N\to\infty}\operatorname{d}(\epsilon^N_t,\,\rho^N_t)=0, \qquad \forall\,t\in (0,\,T].
	\label{eq:approach_sol}
\end{equation}

Formally speaking, we will operate in the setting of the $1$-Wasserstein distance $W_1$, whose definition is as follows:
\begin{equation}
	W_1(\mu,\,\nu)=\inf_{\pi\in\Pi(\mu,\,\nu)}\int_{\R^d\times\R^d}\abs{x-y}\,d\pi(x,\,y),
		\qquad \mu,\nu\in\MNone,
	\label{eq:W1}
\end{equation}
where $\Pi(\mu,\,\nu)$ is the set of all transference plans between the measures $\mu$ and $\nu$, i.e., every $\pi\in\Pi(\mu,\,\nu)$ is a measure on the product space $\R^d\times\R^d$ with marginals $\mu$, $\nu$: $\pi(E\times\R^d)=\mu(E)$, $\pi(\R^d\times E)=\nu(E)$ for every $E\in\cB(\R^d)$. By Kantorovich duality, cf. e.g.,~\cite{villani2009BOOK}, $W_1$ admits also the representation
\begin{equation}
	W_1(\mu,\,\nu)=\sup_{\phi\in\Lipone}\int_{\R^d}\phi\,d(\nu-\mu),
	\label{eq:W1_dual}
\end{equation}
where $\Lipone$ is the space of Lipschitz continuous functions $\phi:\R^d\to\R$ with at most unitary Lipschitz constant. We will use indifferently either expression of $W_1$ depending on the context.

In the following, from Section~\ref{sec:hyp_non-stat} to Section~\ref{sec:scaling}, we recall general results about the solution to~\eqref{eq:cauchy-problem} independently of the geometric structure of the measure. Such results will allow us to discuss, later in Section~\ref{sec:back_discr_cont}, the limits~\eqref{eq:approach_init}-\eqref{eq:approach_sol} previously introduced. 

\subsection{Modeling hypotheses}
\label{sec:hyp_non-stat}
Following the theory developed in~\cite{cristiani2014BOOK,piccoli2013AAM,tosin2011NHM}, we assume some smoothness of the transport velocity. Specifically:

\begin{assumption}[Lipschitz continuity of $v$]
There exist $\Lip(\vd),\,\Lip(K)>0$ such that
\[ \abs{\vd(y)-\vd(x)}\leq\Lip(\vd)\abs{y-x}, \quad \abs{K(y)-K(x)}\leq\Lip(K)\abs{y-x} \]
for all $x,\,y\in\R^d$.
\label{ass:v}
\end{assumption}

Using the expression~\eqref{eq:velocity-transport-eq} of the transport velocity, it is immediate to check that Assumption~\ref{ass:v} implies
\[ \abs{v[\nu](y)-v[\mu](x)}\leq (\Lip(\vd)+N\Lip(K))\abs{y-x}+\Lip(K)W_1(\mu,\,\nu) \]
for all $x,\,y\in\R^d$ and $\mu,\,\nu\in\MNone$, hence letting
\begin{equation}
	\xi^N:=2\max\{\Lip(\vd),\,N\Lip(K)\}
	\label{eq:xi-N}
\end{equation} 
and recalling that we are considering $N\geq 1$ we finally have
\begin{equation}
	\abs{v[\nu](y)-v[\mu](x)}\leq \xi^N\left(\abs{y-x}+\frac{1}{N}W_1(\mu,\,\nu)\right).
	\label{eq:Lip_v}
\end{equation}
 
\subsection{Continuous dependence on the initial datum}
The basic tool for the subsequent analysis is the continuous dependence of the solution to~\eqref{eq:cauchy-problem} on the initial datum, which can be proved using the representation formula~\eqref{eq:pushfwd} based on the flow map~\eqref{eq:def-flow-map}. The result itself and the analytical technique to obtain it are classical in the theory of problem~\eqref{eq:cauchy-problem}, see e.g.,~\cite{dobrushin1979FAA,golse2003JEDP,ha2009CMS}. They have also been extensively used to analyze swarming models, see e.g.,~\cite{canizo2011M3AS} and the thorough review~\cite{carrillo2014CISM}. However, in the present case it is crucial to obtain the explicit dependence on $N$ of some constants appearing in the final estimates, which is less classical due to the fact that here the total mass of the system is not $1$ but  $N$. For this reason, in order to make the paper self-contained, we detail in Appendix~\ref{app:proofs} the proofs of the  following two  results:

\begin{lemma}[Regularity of the flow map]\hfill
\begin{enumerate}[(i)]
\item \label{lemma:gamma_Lip} For all $x,\,y\in\R^d$ and $t\in [0,\,T]$ it results
\[ \abs{\gamma_t(y)-\gamma_t(x)}\leq e^{\xi^Nt}\abs{y-x}. \]
\item \label{lemma:gamma_mu-nu} Let $\mu_\cdot,\,\nu_\cdot\in C([0,\,T];\,\MNone)$ be two solutions to~\eqref{eq:cauchy-problem} with respective initial conditions $\bar{\mu},\,\bar{\nu}\in\MNone$. Call $\gamma^\mu$, $\gamma^\nu$ the flow maps associated to either solution. Then for all $x\in\R^d$ and $t\in [0,\,T]$ it results
\[ \abs{\gamma^\nu_t(x)-\gamma^\mu_t(x)}\leq\frac{\xi^N e^{\xi^Nt}}{N}\int_0^t W_1(\mu_s,\,\nu_s)\,ds. \] 
\end{enumerate}
\label{lemma:reg_gamma}
\end{lemma}

\begin{proposition}[Continuous dependence]
Let $\mu_\cdot,\,\nu_\cdot\in C([0,\,T];\,\MNone)$ be two solutions to~\eqref{eq:cauchy-problem} corresponding to initial conditions $\bar{\mu},\,\bar{\nu}\in\MNone$. Then
\[ W_1(\mu_t,\,\nu_t)\leq e^{\xi^Nt\left(1+e^{\xi^NT}\right)}W_1(\bar{\mu},\,\bar{\nu}),
	\qquad \forall\,t\in (0,\,T]. \]
\label{prop:cont_dep}
\end{proposition}

\subsection{Sequences of measures with growing mass: scaling of the interactions}
\label{sec:scaling}
Let us now consider two sequences of initial conditions of growing mass, say $\{\bar{\mu}^N\}_{N=1}^{\infty},\,\{\bar{\nu}^N\}_{N=1}^{\infty}$ with $\bar{\mu}^N,\,\bar{\nu}^N\in\MNone\ \forall\,N\geq 1$, and the related sequences of solutions to~\eqref{eq:cauchy-problem}, $\{\mu^N_\cdot\}_{N=1}^{\infty},\,\{\nu^N_\cdot\}_{N=1}^{\infty}$ with $\mu^N_\cdot,\,\nu^N_\cdot\in C([0,\,T];\,\MNone)\ \forall\,N\geq 1$. From Proposition~\ref{prop:cont_dep} we have
\begin{equation}
	W_1(\mu^N_t,\,\nu^N_t)\leq e^{\xi^Nt\left(1+e^{\xi^NT}\right)}W_1(\bar{\mu}^N,\,\bar{\nu}^N),
	\label{eq:apriori_N}
\end{equation}
where the exponential factor estimates the amplification at time $t>0$ of the distance between the initial data. If $N$ is sufficiently large then from~\eqref{eq:xi-N} we have $\xi^N=2N\Lip(K)$. In order for the exponential factor to remain bounded for growing $N$ and ensure that $W_1(\bar{\mu}^N,\,\bar{\nu}^N)$ and $W_1(\mu^N_t,\,\nu^N_t)$ are of the same order of magnitude for all $N$, it is necessary that
\begin{equation}
	\Lip(K)=O(N^{-1}) \quad \text{for} \quad N\to\infty.
	\label{eq:LipK_N}
\end{equation}
In other words, we have to suitably scale  the interaction kernel 
%
with respect to $N$. In particular, given a Lipschitz continuous function $\cK:\R^d\to\R^d$ compactly supported in $B_R(0)\subset\R^d$, to comply with~\eqref{eq:LipK_N} we consider the following two-parameter family of interaction kernels:
\begin{equation}
	K(z)=K^N_{\alpha,\beta}(z):=\frac{1}{N^\alpha}\cK\left(\frac{z}{N^\beta}\right), \qquad \alpha,\,\beta\in\R,
	\label{eq:KN}
\end{equation}
whose Lipschitz constant is
\[ \Lip(K^N_{\alpha,\beta})=\frac{\Lip(\cK)}{N^{\alpha+\beta}}. \] 
Clearly, they satisfy~\eqref{eq:LipK_N} as long as
\begin{equation}
	\alpha+\beta\geq 1.
	\label{eq:alpha.beta}
\end{equation}

\begin{example}[Role of $\alpha$]\label{ex:alpha}
An admissible interaction kernel in the family~\eqref{eq:KN} is  obtained for $\alpha=1$, $\beta=0$, i.e.,
\[ K^N_{1,0}(z)=\frac{1}{N}\cK(z). \]
This corresponds to a decreasing interpersonal repulsion when the number of pedestrians increases. Notice that this is the same scaling adopted in the mean field limit~\cite{carrillo2009KRM,carrillo2010MSSET}. In this case, pedestrian velocity~\eqref{eq:velocity-transport-eq} reads
\[ v[\mu^N_t](x)=\vd(x)-\frac{1}{N}\int_{\R^d}\cK(y-x)\,d\mu^N_t(y). \]
Considering that $\mu^N_t(\R^d)=N$, the desired velocity and the interactions have commensurable weights for every $N$.
\end{example}

\begin{example}[Role of $\beta$]\label{ex:beta}
An interaction kernel somehow opposite to $K^N_{1,0}$    (cf. Example~\ref{ex:alpha}) is obtained for $\alpha=0$, $\beta=1$, i.e.,
\[ K^N_{0,1}(z)=\cK\left(\frac{z}{N}\right), \]
whose support is contained in the ball $B_{NR}(0)$. This kernel corresponds to pedestrians who interact with
an increasing number of
 individuals as the total number of people increases. The resulting velocity:
\[ v[\mu^N_t](x)=\vd(x)-\int_{\R^d}\cK\left(\frac{y-x}{N}\right)\,d\mu^N_t(y) \]
is such that the component due to interactions tends to predominate over the desired one for growing $N$.
\end{example}

\begin{example}
Besides the extreme cases in Examples~\ref{ex:alpha},~\ref{ex:beta}, we may also consider intermediate cases in which both $\alpha$ and $\beta$ are simultaneously nonzero (and not necessarily positive). For instance, for $\alpha=2$, $\beta=-1$ we get
\[ K^N_{2,-1}(z)=\frac{1}{N^2}\cK(Nz), \]
which corresponds to interactions that weaken and that are restricted to a contracting 
 sensory region (in fact $\supp{K^N_{2,-1}}\subseteq B_{R/N}(0)$) as $N$ grows. This models pedestrians who agree to stay closer and closer as their number increases, like e.g., in highly crowded train or metro stations during rush hours. The resulting velocity:
\[ v[\mu^N_t](x)=\vd(x)-\frac{1}{N^2}\int_{\R^d}\cK(N(y-x))\,d\mu^N_t(y) \]
is such that $\vd$ dominates for large $N$, meaning that at high crowding pedestrians tend to be passively transported by the flow without interacting.
\end{example}

A scaling of  interactions of type~\eqref{eq:KN} is proposed e.g. in~\cite{capasso2009SAA}. However, it involves only one parameter $\gamma\in (0,\,1)$, corresponding to setting  $\alpha=-\gamma$ and $\beta=-\frac{\gamma}{d}$ in~\eqref{eq:LipK_N}. As this condition does not satisfy~\eqref{eq:alpha.beta}, it has to be regarded as a complementary case, not covered by the theory developed here.

\subsubsection{Scaling equivalence}
Solutions to~\eqref{eq:cauchy-problem} with interaction kernel~\eqref{eq:KN} for different values of the parameters $\alpha$, $\beta$ account for different interpersonal attitudes of pedestrians in congested crowd regimes. Hence, we may expect significantly different solutions for large $N$. Nonetheless, in the case $\vd\equiv 0$ a one-to-one correspondence among them exists, up to a transformation of the space variable depending on $N$. In order to prove it we preliminarily introduce the following

\begin{proposition}
Let $U:\R^d\to\R^d$ be the linear scaling of the space
\[ U(z)=az, \qquad a\in\R \]
and let $\hat{v}$, $\tilde{v}$ be the velocities~\eqref{eq:velocity-transport-eq} computed with the following interaction kernels:
\[ \hat{K}(z)=(\cK\circ U^{-1})(z)=\cK\left(\frac{z}{a}\right), \qquad
	\tilde{K}(z)=(U^{-1}\circ\cK)(z)=\frac{1}{a}\cK(z), \]
where $\cK:\R^d\to\R^d$ is Lipschitz continuous. Assume moreover $\vd\equiv 0$.
\begin{enumerate}
\item[(i)] For all $\mu\in\MNone$ it results
\[ \tilde{v}[\mu](x)=\frac{1}{a}\hat{v}[U\#\mu](ax). \]
\end{enumerate}

Let $\hat{\mu}_\cdot,\,\tilde{\mu}_\cdot\in C([0,\,T];\,\MNone)$ be the solutions to~\eqref{eq:cauchy-problem} with velocities $\hat{v}$, $\tilde{v}$, respectively, and initial conditions such that
\[ \hat{\mu}_0=U\#\tilde{\mu}_0. \]
\begin{enumerate}
\item[(ii)] The flow maps $\hat{\gamma}_t$, $\tilde{\gamma}_t$ correspond to one another as
\[ \hat{\gamma}_t=U\circ\tilde{\gamma}_t\circ U^{-1}, \qquad \forall\,t\in (0,\,T]. \]
\item[(iii)] The solutions satisfy
\[ \hat{\mu}_t=U\#\tilde{\mu}_t, \qquad \forall\,t\in (0,\,T]. \]
\end{enumerate}
\label{prop:corr_sol}
\end{proposition}
\begin{proof}
\begin{enumerate}[(i)]
\item By direct calculation we find
\begin{align*}
	\frac{1}{a}\hat{v}[U\#\mu](ax) &= -\frac{1}{a}\int_{\R^d}\hat{K}(y-ax)\,d(U\#\mu)(y) \\
	& =-\frac{1}{a}\int_{\R^d}\cK\left(\frac{y-ax}{a}\right)\,d(U\#\mu)(y) \\
	& =-\frac{1}{a}\int_{\R^d}\cK(z-x)\,d\mu(z) \\
	& =-\int_{\R^d}\tilde{K}(z-x)\,d\mu(z)=\tilde{v}[\mu](x).
\end{align*}
\item We check that
\[ \hat{\gamma}_t(x):=(U\circ\tilde{\gamma}_t\circ U^{-1})(x)=a\tilde{\gamma}_t\left(\frac{x}{a}\right) \]
complies with the definition~\eqref{eq:def-flow-map}. Using~\eqref{eq:def-flow-map} for $\tilde{\gamma}_t$, we write
\[ \hat{\gamma}_t(x)=x+a\int_0^t\tilde{v}[\tilde{\gamma}_s\#\tilde{\mu}_0]\left(\tilde{\gamma}_s\left(\frac{x}{a}\right)\right)\,ds. \]
Next we observe that
\[ \tilde{\gamma}_s\#\tilde{\mu}_0=(\tilde{\gamma}_s\circ U^{-1}\circ U)\#\tilde{\mu}_0
	=(\tilde{\gamma}_s\circ U^{-1})\#\hat{\mu}_0, \]
hence from the previous point (i) we deduce
\begin{align*}
	\hat{\gamma}_t(x) &= x+a\int_0^t\frac{1}{a}\hat{v}[U\#(\tilde{\gamma}_s\circ U^{-1})\#\hat{\mu}_0]\left(a\tilde{\gamma}_s\left(\frac{x}{a}\right)\right)\,ds \\
	& =x+\int_0^t\hat{v}[\hat{\gamma}_s\#\hat{\mu}_0](\hat{\gamma}_s(x))\,ds,
\end{align*}
as desired.
\item Due to the result in (ii) we have
\[ \hat{\mu}_t=\hat{\gamma}_t\#\hat{\mu}_0=(U\circ\tilde{\gamma}_t\circ U^{-1})\#(U\#\tilde{\mu}_0)
	=(U\circ\tilde{\gamma}_t)\#\tilde{\mu}_0=U\#\tilde{\mu}_t. \qedhere \]
\end{enumerate}
\end{proof}

As a consequence of Proposition~\ref{prop:corr_sol}, we can prove a correspondence among the dynamics governed by different interaction kernels of the family~\eqref{eq:KN}.

\begin{theorem}[Scaling equivalence]
Let $\mu^N_\cdot,\,\nu^N_\cdot\in C([0,\,T];\,\MNone)$ be the solutions to~\eqref{eq:cauchy-problem} corresponding to interaction kernels $K^N_{\alpha,\beta}$, $K^N_{\alpha',\beta'}$, with
\[ \alpha+\beta=\alpha'+\beta', \]
and to initial conditions $\bar{\mu}^N,\,\bar{\nu}^N=U^N\#\bar{\mu}^N\in\MNone$, respectively, where
\[ U^N(z)=N^{\beta'-\beta}z. \]
Then
\[ \nu^N_t=U^N\#\mu^N_t, \qquad \forall\,t\in (0,\,T]. \]
\end{theorem}
\begin{proof}
By~\eqref{eq:KN} we have, on the one hand,
\[ K^N_{\alpha,\beta}(z)=\frac{1}{N^\alpha}\cK\left(\frac{z}{N^\beta}\right)=
	\frac{1}{N^{\beta'-\beta}}\cdot\frac{1}{N^{\alpha'}}\cK\left(\frac{z}{N^\beta}\right), \]
where we have used the fact that $\alpha=\alpha'+\beta'-\beta$, and, on the other hand,
\[ K^N_{\alpha',\beta'}(z)=\frac{1}{N^{\alpha'}}\cK\left(\frac{z}{N^{\beta'}}\right)=
	\frac{1}{N^{\alpha'}}\cK\left(\frac{1}{N^{\beta'-\beta}}\cdot\frac{z}{N^\beta}\right). \]
Thus:
\begin{align*}
	K^N_{\alpha,\beta}(z) &= \frac{1}{a}\cdot\frac{1}{N^{\alpha'}}\cK\left(\frac{z}{N^\beta}\right) \\
	K^N_{\alpha',\beta'}(z) &= \frac{1}{N^{\alpha'}}\cK\left(\frac{1}{a}\cdot\frac{z}{N^\beta}\right)
\end{align*}
for $a=N^{\beta'-\beta}$ and the thesis follows from Proposition~\ref{prop:corr_sol}.
\end{proof}

\subsection{Back to discrete and continuous models}
\label{sec:back_discr_cont}
Conditions~\eqref{eq:KN}-\eqref{eq:alpha.beta} imply
\[ \xi^N\leq\xi^\ast:=2\max\{\Lip(\vd),\,\Lip(\cK)\}, \qquad \forall\,N\geq 1. \]
When using~\eqref{eq:apriori_N} for sequences of discrete and continuous measures we can incorporate this fact and write
\begin{equation}
	W_1(\epsilon^N_t,\,\rho^N_t)\leq e^{\xi^\ast t\left(1+e^{\xi^\ast T}\right)}W_1(\bar{\epsilon}^N,\,\bar{\rho}^N),
		\qquad \forall\,t\in (0,\,T],\ \forall\,N\geq 1.
	\label{eq:W1.discr-cont}
\end{equation}
Hence mutually approaching sequences of discrete and continuous solutions to~\eqref{eq:cauchy-problem}, cf.~\eqref{eq:approach_sol}, are possible, provided one is able to construct mutually approaching sequences of initial conditions at the corresponding scales, cf.~\eqref{eq:approach_init}. In the following we discuss a possible procedure leading to the desired result.

Let $\bar{X}^1,\,\dots,\,\bar{X}^N\in\R^d$ be the initial positions of $N$ distinct microscopic pedestrians. The associated discrete distribution $\bar{\epsilon}^N$ is constructed from~\eqref{eq:construction-of-meas-epsilon}, while, given $r>0$, we define
\begin{equation}
	\bar{\rho}^N=\sum_{i=1}^{N}\rho^i, \qquad
		\rho^i(x)=\frac{1}{r^d}f\left(\frac{x-\bar{X}^i}{r}\right),
	\label{eq:init.rhoN}
\end{equation}
where $f:\R^d\to\R$ is a nonnegative function such that
\begin{equation}
	\supp{f}\subseteq\overline{B_1(0)}, \qquad \int_{B_1(0)}f(x)\,dx=1.
	\label{eq:f}
\end{equation}
Consequently, $\supp{\rho^i}\subseteq\overline{B_r(\bar{X}^i)}$ and moreover $\rho^i(B_r(\bar{X}^i))=1$ each $i$, thus $\bar{\rho}^N$ is the superposition of $N$ piecewise constant density bumps, each of which carries a unit mass representative of one pedestrian. We assume 
\begin{equation}
	r<\frac{1}{2}\min_{i\ne j}\abs{\bar{X}^j-\bar{X}^i}
	\label{eq:disc-min-radius}
\end{equation}
which ensures no overlapping of the supports of the $\rho^i$'s.

In the remaining part of this section we compute the Wasserstein distance between $\bar{\epsilon}^N$ and $\bar{\rho}^N$ and we study its trend with respect to $N$.

\begin{proposition}[Distance between the initial conditions]
Let
\[ m_f:=\int_{B_1(0)}\abs{x}f(x)\,dx. \]
The distance between $\bar\epsilon^N$ and $\bar\rho^N$ is
\[ W_1(\bar{\epsilon}^N,\,\bar{\rho}^N)=m_fNr. \]
\label{prop:W1.init}
\end{proposition}
\begin{proof}
Here we use the expression~\eqref{eq:W1_dual} of $W_1$. Let us consider the case $N=1$ first. Since $\bar{\epsilon}^1$ is a single Dirac mass, there is no ambiguity in the construction of the optimal transference plan $\pi$ between $\bar{\epsilon}^1$ and $\bar{\rho}^1$ (i.e., the transference plan which realizes the infimum in~\eqref{eq:W1}), which is just the tensor product of the measures:
\[ \pi(x,\,y)=\bar{\rho}^1(x)\otimes\bar{\epsilon}^1(y). \]
By substitution in~\eqref{eq:W1} we find
\begin{align*}
	W_1(\bar{\epsilon}^1,\,\bar{\rho}^1) &= \int_{\R^d\times\R^d}\abs{x-y}\,d(\bar{\rho}^1(x)\otimes\bar{\epsilon}^1(y)) \\
	& =\int_{B_r(\bar{X}^1)}\abs{x-\bar{X}^1}\rho^1(x)\,dx \\
	& =\frac{1}{r^d}\int_{B_r(\bar{X}^1)}\abs{x-\bar{X}^1}f\left(\frac{x-\bar{X}^1}{r}\right)\,dx \\
	& =r\int_{B_1(0)}\abs{y}f(y)\,dy & \text{(set $y:=(x-\bar{X}^1)/r$)} \\
	& =m_fr.
\end{align*}

We pass now to characterize transference plans between $\bar{\epsilon}^N$ and $\bar{\rho}^N$ in the case $N>1$. Every element of the continuous mass $\bar{\rho}^N$ is transported onto a delta, a condition that, in the spirit of the so-called \emph{semi-discrete Monge-Kantorovich problem}~\cite{abdellaoui1998JCAM}, can be expressed by a measure $\pi$ on $\R^d\times\R^d$ of the form
\begin{equation}
	\pi(x,\,y)=\bar{\rho}^N(x)\otimes\sum_{j=1}^{N}\alpha_j(x)\delta_{\bar{X}^j}(y),
	\label{eq:representation-splitting-mass}
\end{equation}
where, in order to ensure conservation of the mass, the $\alpha_j$'s are such that
\begin{equation}
	\alpha_j(x)\geq 0, \quad \sum_{j=1}^{N}\alpha_j(x)=1,
		\qquad \forall\,j=1,\,\dots,\,N,\ \forall\,x\in\supp{\bar{\rho}^N}.
	\label{eq:alphaj.1}
\end{equation}
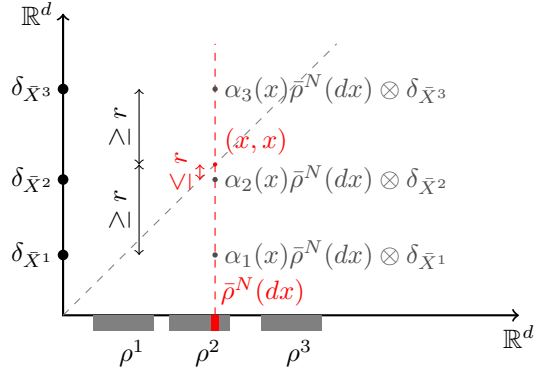
\begin{figure}[!t]
\centering
\begin{tikzpicture}[scale=2, Name/.style = {font={\bfseries}}]
    \draw [<->,thick] (0,2) node (yaxis) [left] {$\R^d$}
        |- (3,0) node (xaxis) [below] {$\R^d$};

    \draw[gray,dashed] (0,0) coordinate (origin) -- (1.8,1.8) coordinate (diag_estr);

    \coordinate (d1) at (0,0.4);
    \coordinate (d2) at (0,0.9);
    \coordinate (d3) at (0,1.5);

    \fill[black] (d1) node [left] {$\delta_{\bar{X}^1}$} circle (1pt);   
    \fill[black] (d2) node [left] {$\delta_{\bar{X}^2}$} circle (1pt);
    \fill[black] (d3) node [left] {$\delta_{\bar{X}^3}$} circle (1pt);
    \fill[gray] (.2,0) coordinate (r_11) rectangle (.6,-.1) node [below left,text=black] {$\rho^1$};
    \fill[gray] (.7,0) coordinate (r_21) rectangle (1.1,-.1) node [below left,text=black] {$\rho^2$};
    \fill[gray] (1.3,0) coordinate (r_31) rectangle (1.7,-.1) node [below left,text=black] {$\rho^3$};

    \coordinate (xcoord) at (1,0);
    \fill[red] ($ (xcoord) + (-.025,0) $) node (xval) [above right] {$\bar\rho^N(dx)$} rectangle ($(xcoord)+(.025,-.1)$);
    \draw[red,dashed] (xcoord) -- ($(xcoord)+(0,1.8)$) coordinate (red_line_top);
    \fill[white!30!black] ($(xcoord)+(d1)$) coordinate (delta1_tens) node [right] {$\alpha_1(x)\bar\rho^N(dx)\otimes\delta_{\bar{X}^1}$} circle (.5pt);
    \fill[white!30!black] ($(xcoord)+(d2)$) coordinate (delta2_tens) node [right] {$\alpha_2(x)\bar\rho^N(dx)\otimes\delta_{\bar{X}^2}$} circle (.5pt);
    \fill[white!30!black] ($(xcoord)+(d3)$) coordinate (delta3_tens) node [right] {$\alpha_3(x)\bar\rho^N(dx)\otimes\delta_{\bar{X}^3}$} circle (.5pt);
 
    \coordinate (xx) at (intersection of origin--diag_estr and xcoord--red_line_top);
    \node [above right,text=red] at (xx) {$(x,x)$};
    \fill[red] (xx) circle (.4pt);
    
    \coordinate (offsetd1) at (-0.5,0);
    \coordinate (offsetd2) at (-0.1,0);

    \draw [<->] ($(xx)+(offsetd1)$) -- ($(delta1_tens)+(offsetd1)$) node [pos=0.5,sloped,above] {$\geq r$};
    \draw [<->] ($(xx)+(offsetd1)$) -- ($(delta3_tens)+(offsetd1)$) node [pos=0.5,sloped,above,rotate=180] {$\geq r$};
    \draw [<->,red] ($(xx)+(offsetd2)$) -- ($(delta2_tens)+(offsetd2)$) node [pos=0.5,sloped,above,text=red] {$\leq r$};
\end{tikzpicture}
\caption{Transference plans~\eqref{eq:representation-splitting-mass} in $\R^d\times\R^d$ for $N=3$. Each continuous mass element is within a distance $r$ from its corresponding Dirac mass. Therefore, owing to~\eqref{eq:disc-min-radius}, it is farther than $r$ from all other Dirac masses.}
\label{fig-reference-splitting}
\end{figure} 
The representation~\eqref{eq:representation-splitting-mass} means that the infinitesimal element of continuous mass $d\bar{\rho}^N(x)$ located in $x\in\supp{\bar{\rho}^N}$ is split in the points $\{\bar{X}^j\}_{j=1}^{N}$ following the convex combination given by the coefficients $\alpha_{j}(x)$ (cf. Fig.~\ref{fig-reference-splitting}). The measure $\pi$ is generally not a transference plan between $\bar{\epsilon}^N$ and $\bar{\rho}^N$, because it is not guaranteed to have marginal $\bar{\epsilon}^N$ with respect to $y$. In particular, a non-unit mass might be allocated in every $\bar{X}^j$. In order to have a transference plan, the further condition
\begin{equation}
	\int_{\R^d}\alpha_j(x)\,d\bar{\rho}^N(x)=1, \qquad \forall\,j=1,\,\dots,\,N
	\label{eq:alphaj.2}
\end{equation}
needs to be enforced.

Let us consider, for a transference plan of the form~\eqref{eq:representation-splitting-mass} with $\bar{\rho}^N$ as in~\eqref{eq:init.rhoN}, the global transportation cost:
\begin{align*}
	\int_{\R^d\times\R^d}\abs{x-y}\,d\pi(x,\,y) &= \int_{\R^d\times\R^d}\abs{x-y}\,d\left(\sum_{i=1}^{N}\rho^i(x)\otimes\sum_{j=1}^{N}\alpha_j(x)\delta_{\bar{X}^j}(y)\right) \\
	& =\sum_{i,j=1}^{N}\int_{\R^d\times\R^d}\abs{x-y}\,d(\rho^i(x)\otimes\alpha_j(x)\delta_{\bar{X}^j}(y)) \\
	&= \sum_{i=1}^{N}\int_{B_r(\bar{X}^i)}\sum_{j=1}^{N}\alpha_j(x)\abs{x-\bar{X}^j}\,d\rho^i(x).
\end{align*}
Owing to~\eqref{eq:alphaj.1} we have
\begin{align*}
	\sum_{j=1}^{N}\alpha_j(x)\abs{x-\bar{X}^j} &= \sum_{j\ne i}\alpha_j(x)\abs{x-\bar{X}^j}+\left(1-\sum_{j\ne i}\alpha_j(x)\right)\abs{x-\bar{X}^i}
	\intertext{whence, for $x\in B_r(\bar{X}^i)$ and taking~\eqref{eq:disc-min-radius} into account,}
	& \geq r\sum_{j\ne i}\alpha_j(x)+\abs{x-\bar{X}^i}-r\sum_{j\ne i}\alpha_j(x) \\
	& =\abs{x-\bar{X}^i},
\end{align*}
thus ultimately
\begin{align*}
	\int_{\R^d\times\R^d}\abs{x-y}\,d\pi(x,\,y) &\geq \sum_{i=1}^{N}\int_{B_r(\bar{X}^i)}\abs{x-\bar{X}^i}\,d\rho^i(x) \\
	& =\sum_{i=1}^{N}\int_{\R^d\times\R^d}\abs{x-y}\,d(\rho^i(x)\otimes\delta_{\bar{X}^i}(y)).
\end{align*}

Notice that the transference plan
\[ \pi=\sum_{i=1}^{N}\rho^i\otimes\delta_{\bar{X}^i} \]
is of the form~\eqref{eq:representation-splitting-mass} for, e.g., the coefficients $\alpha_j(x)=\chi_{B_r(\bar{X}^j)}(x)$ which fulfill both~\eqref{eq:alphaj.1} and~\eqref{eq:alphaj.2}. The previous calculation says that it is actually the optimal transference plan between $\bar{\epsilon}^N$ and $\bar{\rho}^N$, i.e., the one which ensures the minimum transportation cost. Thus
\[ W_1(\bar{\epsilon}^N,\,\bar{\rho}^N)=\sum_{i=1}^{N}\int_{\R^d\times\R^d}\abs{x-y}\,d(\rho^i(x)\otimes\delta_{\bar{X}^i}(y))
	=NW_1(\bar{\epsilon}^1,\,\bar{\rho}^1), \]
whence the thesis follows.
\end{proof}

Thanks to Proposition~\ref{prop:W1.init} we finally obtain the main result of the paper:

\begin{theorem}[Discrete-continuous convergence]
Let $\bar{\epsilon}^N\in\MNone$ be given and let $\bar{\rho}^N\in\MNone$ be constructed as in~\eqref{eq:init.rhoN}-\eqref{eq:f}. Set $r=c_NN^{-\gamma}$, $\gamma>1$, where $0<c_N\leq 1$ is possibly used to enforce condition~\eqref{eq:disc-min-radius}. Let moreover the interaction kernel satisfy~\eqref{eq:KN}-\eqref{eq:alpha.beta}. Then
\[ \lim_{N\to\infty} W_1(\epsilon^N_t,\,\rho^N_t)=0, \qquad \forall\,t\in (0,\,T]. \]
\label{theo:conv.discr-cont}
\end{theorem}
\begin{proof}
From Proposition~\ref{prop:W1.init} with the given $r$ we obtain:
$$ W_1(\bar{\epsilon}^N,\,\bar{\rho}^N)=\frac{m_fc_N}{N^{\gamma-1}}\xrightarrow{N\to\infty} 0 $$
because $\gamma>1$. Hence the thesis follows from~\eqref{eq:W1.discr-cont}.
\end{proof}

It is worth remarking that, because of~\eqref{eq:disc-min-radius}, in dimension $d$ a bound on the radius $r$ of the form $r<CN^{-1/d}$, where $C>0$ is a constant, holds true e.g., when considering homogeneous distributions of pedestrians in bounded domains (cf. the next example of regular lattices). This is not sufficient by itself to comply with the hypotheses of Theorem~\ref{theo:conv.discr-cont}, but choosing $r$ as
\begin{equation}
	r=CN^{-(1+h)/d}, \qquad N>1,
	\label{eq:r}
\end{equation}
where $h>0$, is instead sufficient if
\[ h>d-1. \]
Then, according to~\eqref{eq:W1.discr-cont} and to Proposition~\ref{prop:W1.init}, the distance between the discrete and continuous solutions to~\eqref{eq:cauchy-problem} scales with $N$ as
\[ W_1(\epsilon^N_t,\,\rho^N_t)\leq O\left(N^{(d-1-h)/d}\right) \quad \text{for} \quad N\to\infty. \]

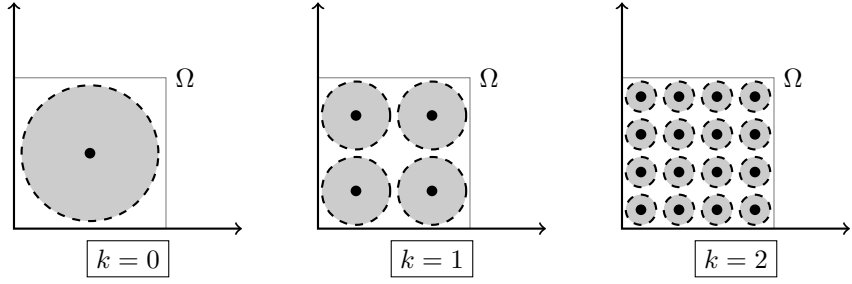
\begin{figure}[!t]
\centering
\begin{tikzpicture}[scale=2,Name/.style={font={\bfseries}}]
	\draw[gray] (origin) rectangle (1,1);
	\draw  [<->,thick] (0,1.5) node (yaxis) [left] {}
        |- (1.5,0) node (xaxis) [below] {};
	\fill[gray!40!white] (.5,.5) circle (.45);
	\draw[black,thick,dashed] (.5,.5) circle (.45);
	\fill[black] (.5,.5) node [below] {} circle (1pt);

	\node[draw=none,right] at (1,1) {$\Omega$};
	\node[draw,align=left] at (.75,-.2) {$k=0$};
	
	\draw[gray] (2,0)  rectangle (3,1) ;
	\draw[<->,thick] (2,1.5) node (yaxis) [left] {}
       |- (3.5,0) node (xaxis) [below] {};

	\foreach \x in {2.25,2.75}
	\foreach \y in {.25,.75} 
		{
			\fill[gray!40!white] (\x,\y) circle (.225);
			\draw[black,thick,dashed] (\x,\y) circle (.225);
			\fill[black] (\x,\y) node [below]  {}  circle (1pt);
		}
	\node[draw=none,right] at (3,1) {$\Omega$};
	\node[draw,align=left] at (2.75,-.2) {$k=1$};
	\draw[gray] (4,0) rectangle (5,1);
	\draw[<->,thick] (4,1.5) node (yaxis) [left] {}
       |- (5.5,0) node (xaxis) [below] {};

	\foreach \x in {4.125,4.375,4.625,4.875}
	\foreach \y in {.125,.375,.625,.875} 
		{
			\fill[gray!40!white] (\x,\y) circle (.1);
			\draw[black,thick,dashed] (\x,\y) circle (.1);
			\fill[black] (\x,\y) node [below] {} circle (1pt);
		}
	\node[draw=none,right] at (5,1) {$\Omega$};
	\node[draw,align=left] at (4.75,-.2) {$k=2$};
\end{tikzpicture}
\caption{Discrete pedestrians (dots) and their continuous counterparts (circular regions delimited by dashed lines) in $\Omega=[0,\,1]^2$ ($d=2$) and for $k=0,\,1,\,2$.}
\label{fig:example-reg-lattice} 
\end{figure}

\begin{example}[Regular lattices]
Homogeneous pedestrian distributions can be obtained, for instance, by considering regular lattices.

Let $\Omega=[0,\,1]^d$. We partition it in $N_k=2^{kd}$ equal hypercubes of edge size $2^{-k}$, then we position $N_k$ discrete pedestrians $\{\bar{X}^i\}_{i=1}^{N_k}$ in their centroids, cf.~Fig.~\ref{fig:example-reg-lattice}, so that
\[ \min_{i\ne j}\abs{\bar{X}^j-\bar{X}^i}=2^{-k}. \]
Owing to~\eqref{eq:disc-min-radius} we need then to take
\[ r<2^{-(k+1)}=\frac{1}{2}N_k^{-1/d}, \]
which, following~\eqref{eq:r}, we satisfy by setting $r=2^{-(1+(h+1)k)}$ for $h>0$. For this value of $r$, let us set
\[ \rho^i(x)=\frac{d}{\omega_dr^d}\chi_{B_r(\bar{X}^i)}(x), \qquad i=1,\,\dots,\,N, \]
where $\omega_d$ is the surface area of the unit ball in $\R^d$ and $\chi$ the characteristic function. These $\rho^i$'s are of the form~\eqref{eq:init.rhoN} for $f(x)=\frac{d}{\omega_d}\chi_{B_1(0)}(x)$, which complies with~\eqref{eq:f} and moreover is such that $m_f=\frac{d}{1+d}$. This entails:
\begin{itemize}
\item in dimension $d=1$,
\[ W_1(\bar{\epsilon}^{N_k},\,\bar{\rho}^{N_k})=2^{-2-hk}, \]
which converges to zero for $k\to\infty$ if $h>0$;
\item in dimension $d=2$,
\[ W_1(\bar{\epsilon}^{N_k},\,\bar{\rho}^{N_k})=\frac{2^{(1-h)k}}{3}, \]
which converges to zero for $k\to\infty$ if $h>1$;
\item in dimension $d=3$,
\[ W_1(\bar{\epsilon}^{N_k},\,\bar{\rho}^{N_k})=3\cdot 2^{(2-h)k-3}, \]
which converges to zero for $k\to\infty$ if $h>2$.
\end{itemize}
\end{example}
 
\section{Discussion}
\label{sec:discussion}
In this paper we have investigated microscopic and macroscopic differential models of systems of interacting particles, chiefly inspired by human crowds, for an increasing number $N$ of total agents. The main novelty was the consideration of \emph{massive} particles, i.e., particles whose mass does not scale with the number $N$. This implies that the continuous model is not obtained in the limit $N\to\infty$ from the discrete model, rather it is postulated \emph{per se} for every value of $N$. 
The question then arises under which conditions  the discrete and continuous models are  counterparts of one another at smaller and larger scales.

In particular, we have proved that the solutions to the following two models:
\begin{align*}
	& \dot{X}^i_t=\vd(X^i_t)+\frac{1}{N^\alpha}\sum_{j=1}^{N}\cK\left(\frac{X^j_t-X^i_t}{N^\beta}\right),
		\quad i=1,\,\dots,\,N \\[2mm]
	& \partial_t\rho_t+\div\left(\rho_t\left(\vd+\frac{1}{N^\alpha}\int_{\R^d}\cK\left(\frac{y-\cdot}{N^\beta}\right)\rho_t(y)\,dy\right)\right)=0,
		\quad \rho_t(\R^d)=N,
\end{align*}
where the velocity field $\vd$ and the interaction kernel $\cK$ are assumed to be Lipschitz continuous, converge to one another in the sense of the $1$-Wasserstein distance when $N\to\infty$ if the parameters $\alpha$, $\beta$ are such that $\alpha+\beta\geq 1$ and if, in addition, the respective initial conditions approximate each other. This fact, which is schematically illustrated in Fig.~\ref{fig:limit_massive}, has  implications from the modeling point of view, especially as far as the role of single scales and their possible coupling are concerned.

\begin{figure}[!t]
\centering
\includegraphics[width=\textwidth]{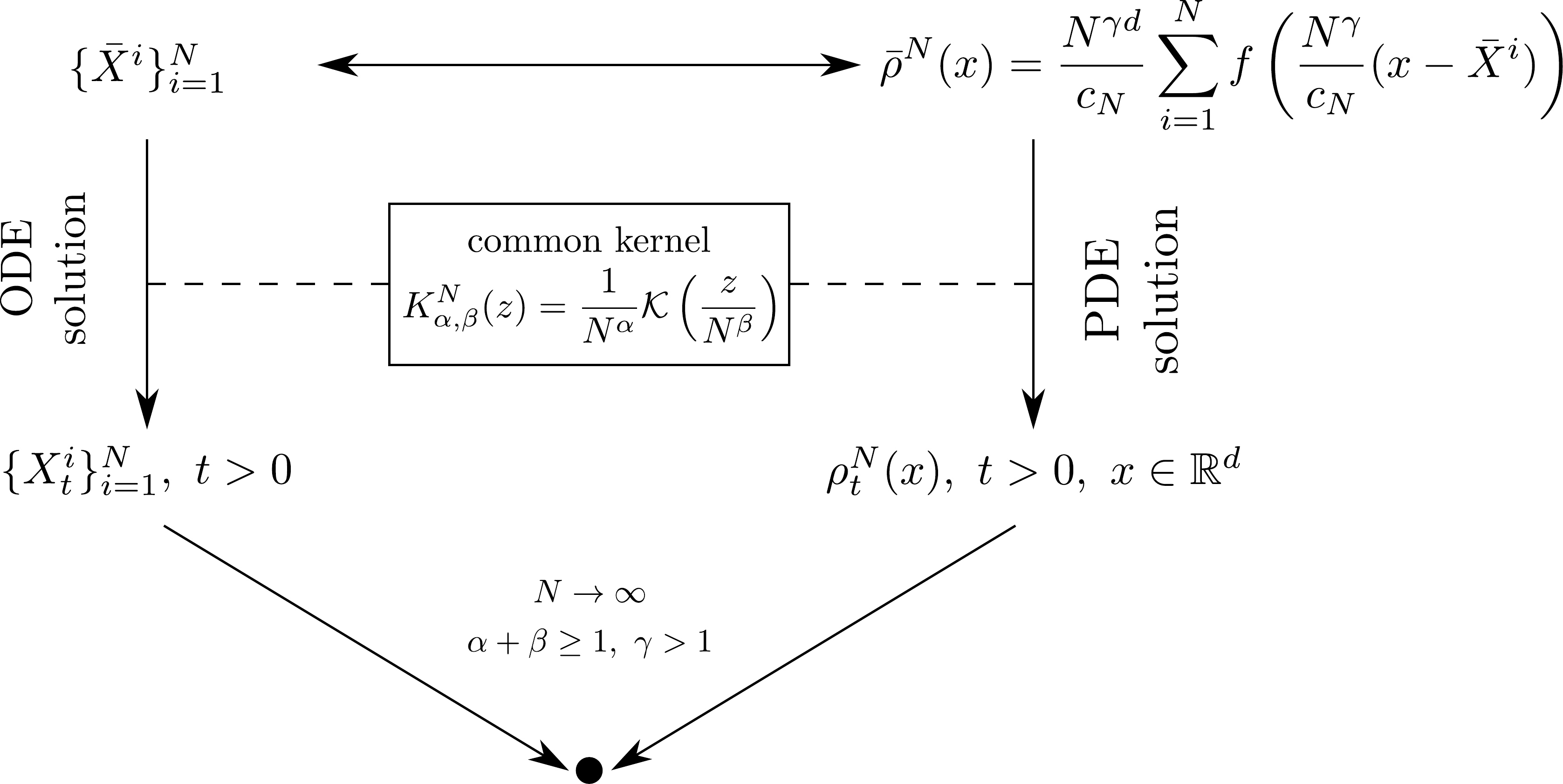}

\caption{A schematic illustration of the result of the paper. The ODE and PDE solutions approach each other for $N\to\infty$ under certain conditions on the scaling of the interactions and of the initial conditions.}
\label{fig:limit_massive}
\end{figure}

First of all, we point out that interactions are modeled in an $N$-dependent way by means of the kernel
\[ K^N_{\alpha,\beta}(z)=\frac{1}{N^\alpha}\cK\left(\frac{z}{N^\beta}\right). \]
More specifically, the function $\cK$ expresses the basic interaction trend (for instance, repulsion) while the factors $N^{-\alpha}$, $N^{-\beta}$ modulate it depending on the total number of particles. This is consistent with the idea that particles like e.g., pedestrians do not behave the same regardless of their number. The strength of their mutual repulsion or the acceptable interpersonal distances may vary considerably from free to congested situations. We model this aspect by 
acting on the values of $\alpha$, $\beta$. Hence, the various scalings contained in the two-parameter family of kernels $K^N_{\alpha,\beta}$ correspond to different \emph{interpersonal attitudes} of the particles for growing $N$. As our analysis demonstrates, the latter need to be taken into account for ensuring consistency of a given interaction model at different scales. In this respect, we have shown that if such a scaling is neglected the discrete and continuous models predict quantitatively different, albeit qualitatively analogous, emerging equilibria. In particular, they yield different asymptotic speeds of the particles, which do not approach each other as $N$ grows.

Second,  our  analysis shows the correct parallelism between first order microscopic and macroscopic models which do not originate from one another but are formulated independently by aprioristic choices of the scales. In this view, the utility of such a parallelism is twofold. On  one hand, it accounts for the \emph{interchangeability of the two models} at sufficiently high numbers of particles, with:
\begin{inparaenum}[(i)]
\item the possibility to switch from a microscopic to a macroscopic description, which may be more convenient for the \emph{a posteriori} calculation of observable quantities and statistics of interest for applications;
\item the possibility to infer qualitative properties at one scale from their rigorous knowledge at the other scale (for instance, the microscopic model can be expected to exhibit qualitatively the same nonlinear diffusive behavior for large $N$ which is typically proved quantitatively for macroscopic conservation laws with nonlocal repulsive flux).
\end{inparaenum}
On the other hand, it allows one to motivate and support \emph{multiscale couplings} of microscopic and macroscopic models~\cite{cristiani2011MMS,cristiani2014BOOK}, which are supposed to provide a \emph{dual representation} of the same particle system at different scales. In this case, the interest does not lie as much in the congested regime (large $N$), where the two models have been proved to be equivalent, but rather in the moderately crowded one, where the discrete and continuous solutions can complement each other with effects which would not be recovered at a single scale.

\appendix

\section{Technical proofs}
\label{app:proofs}

\begin{proof}[Proof of Proposition~\ref{prop:stability_micro}]
First we observe that the measure~\eqref{eq:epsilonN-timet} is a solution to~\eqref{eq:micro-proto}, in fact using~\eqref{eq:equispaced-lattice} in~\eqref{eq:micro-proto} together with Assumption~\ref{ass:K-properties}\eqref{ass:K-compact_supp}-\eqref{ass:K-monotone} we get
\[ \dot{X}^i_t=\vd-\sum_{j>i}K\left(\left(j-i\right)\frac{L}{N}\right), \] 
which reduces to~\eqref{eq:vel-lattice} by setting $h=j-i$.

To show that $\epsilon_t$ is stable and possibly attractive we use a perturbation argument. We define the perturbed positions $\tilde{X}^i_t=X^i_t+\eta^i_t$, then we plug them into~\eqref{eq:micro-proto} to find
\[ \dot{\tilde{X}}^i_t=\vd-\sum_{j=1}^{N}K(X_t^j-X^i_t+\eta^j_t-\eta^i_t). \]
Linearizing this around $X^j_t-X^i_t$ and using Assumption~\ref{ass:K-properties}\eqref{ass:K-compact_supp} gives
\[ \dot{\tilde{X}}^i_t=w-\sum_{j>i}K'\left((j-i)\frac{L}{N}\right)(\eta^j_t-\eta^i_t), \]
which, considering that $\dot{\eta}^i_t=\dot{\tilde{X}}^i_t-w$, further yields
\begin{equation}
	\dot{\eta}^i_t=-\sum_{j>i}K'\left((j-i)\frac{L}{N}\right)(\eta^j_t-\eta^i_t)
		=-\sum_{h=1}^{N-1}K'\left(h\frac{L}{N}\right)(\eta^{i+h}_t-\eta^i_t).
	\label{eq:perturbation-discrete-sum}
\end{equation}

Finally, we claim that $\epsilon_t$ is:
\begin{enumerate}[(a)]
\item \label{stable} stable if $R\leq L/N$;
\item \label{stable_attr} stable and attractive if $R>L/N$.
\end{enumerate}

In case~\eqref{stable} the sum in~\eqref{eq:perturbation-discrete-sum} is zero as $K'(L/N)=K'(2L/N)=\ldots=0$ by Assumption~\ref{ass:K-properties}\eqref{ass:K-compact_supp}. Therefore, small perturbations remain constant over time, which is sufficient to ensure stability.

In case~\eqref{stable_attr} we make the ansatz
\begin{equation}
	\eta^i_t=\sum_{k=1}^{N-1}C_ke^{\sigma_kt+\iunit\frac{2\pi}{N}ki}
		\qquad (\iunit=\textup{imaginary unit}),
	\label{eq:expansion} 
\end{equation}
where $C_k\in\R$, which reflects the periodicity of $\eta^i_t$ with respect to $i$. Notice that the expansion above starts from $k=1$ because the term for $k=0$ is not relevant in the present context, in fact it corresponds simply to a further rigid translation of the $X^i_t$'s. After substituting in~\eqref{eq:perturbation-discrete-sum}, we obtain that~\eqref{eq:expansion} is a solution as long as
\[ \sigma_k=-\sum_{h=1}^{N-1}K'\left(h\frac{L}{N}\right)\left(e^{\iunit\frac{2\pi}{N}kh}-1\right), \]
whence 
\begin{equation} 
	\Real(\sigma_k)=\sum_{h=1}^{N-1}K'\left(h\frac{L}{N}\right)\left(1-\cos\left(\frac{2\pi}{N}kh\right)\right).
	\label{eq:real.part.lambda}
\end{equation}
Since $0<h,\,k<N$, it results $1-\cos(2\pi hk/N)>0$. Because of Assumption~\ref{ass:K-properties}\eqref{ass:K-monotone}, every term of the sum at the right-hand side of~\eqref{eq:real.part.lambda} is either negative (for $hL/N<R$, i.e., within the sensory region) or zero (for $hL/N\geq R$, i.e., outside the sensory region). Moreover, since $R>L/N$, the sum has at least a non-vanishing term (for $h=1$). Therefore $\Real(\sigma_k)<0$ for all $k=1,\,2,\,\dots,\,N-1$ and we have stability and attractiveness.
\end{proof}

\begin{proof}[Proof of Proposition~\ref{prop:stability_macro}]
The constant solution~\eqref{eq:homog-meas} is clearly a solution to~\eqref{eq:macro-proto} when $\vd$ is constant, for then
\[ v[\bar{\rho}]=\vd-\bar{\rho}\int_{x}^{x+R}K(y-x)\,dy=\vd-\bar{\rho}\int_{0}^{R}K(z)\,dz \]
is constant as well. To study its local stability we consider a perturbation of it of the form
\[ \tilde{\rho}_t=\bar{\rho}+\eta\varrho_t \]
for $\eta\in\R$, then we plug it into~\eqref{eq:macro-proto} to have
\[ \partial_t(\bar{\rho}+\eta\varrho_t)+\partial_x((\bar{\rho}+\eta\varrho_t)v[\bar{\rho}+\eta\varrho_t])=0. \]
In the limit of small $\eta$ this gives the following linearized equation for the perturbation:
\begin{equation}
	\partial_t\varrho_t+\partial_x(\bar{\rho}v[\varrho_t]+\varrho_t v[\bar{\rho}])=0,
	\label{eq:linear-stab}
\end{equation} 
for which we make the ansatz of periodic solution in space:
\[ \varrho_t(x)=\sum_{k\in\Z}C_ke^{\sigma_kt+\iunit\frac{2\pi}{L}kx}
	\qquad (\iunit=\textup{imaginary unit}), \]
with $C_k\in\R$. Actually, similarly to the microscopic case in Proposition~\ref{prop:stability_micro}, we can neglect the term of the sum for $k=0$, because again it corresponds to a constant in space perturbation.

By linearity we consider one term of the sum at a time, i.e., we take $\varrho_t(x)=C_ke^{\sigma_kt+\iunit\frac{2\pi}{L}kx}$. Substituting in~\eqref{eq:linear-stab} we find
\[ \sigma_k+\iunit\frac{2\pi}{L}k\left(v[\bar{\rho}]+\bar{\rho}\mathcal{K}_k\right)=0 \quad
	\textup{with} \quad \mathcal{K}_k:=-\int_0^L K(z)e^{\iunit\frac{2\pi}{L}kz}\,dz. \]
The asymptotic trend in time of $\varrho_t$ depends on $\Real(\sigma_k)$, which, according to the previous equation, is given by
\[ \Real(\sigma_k)=\frac{2\pi}{L}k\bar{\rho}\Img(\mathcal{K}_k)=
	-\frac{2\pi}{L}k\bar{\rho}\int_0^L K(z)\sin\left(\frac{2\pi}{L}kz\right)\,dz. \]

We claim that $\Real(\sigma_k)<0$ for all $k\neq 0$.

First, we observe that $\frac{2\pi}{L}k\bar{\rho}\Img(\mathcal{K}_k)$ is even in $k$, since it is the product of two odd functions in $k$. Thus $\Real(\sigma_{-k})=\Real(\sigma_{k})$ and we can confine ourselves to $k>0$. Second, $\Img(\mathcal{K}_k)$ can be written as
\begin{align}
	\begin{aligned}[c]
		\Img(\mathcal{K}_k) &= -\int_0^L K(z)\sin\left(\frac{2\pi}{L}kz\right)\,dz \\
		& =-\sum_{q=0}^{k-1}\int_{qL/k}^{(q+1)L/k}K(z)\sin\left(\frac{2\pi}{L}kz\right)\,dz
	\end{aligned}
	\label{eq:sum-decomposition}
\end{align}
and, in addition, for each term of the sum at the right-hand side it holds
\begin{multline*}
	\int_{qL/k}^{(q+1)L/k}K(z)\sin\left(\frac{2\pi}{L}kz\right)\,dz \\
		=\int_{qL/k}^{(q+1/2)L/k}\left(K(z)-K\left(z+\frac{L}{2k}\right)\right)\sin\left(\frac{2\pi}{L}kz\right)\,dz.
\end{multline*}
In the interval $[qL/k,\,(q+1/2)L/k]$ the sine function is nonnegative. Furthermore, in view of Assumption~\ref{ass:K-properties}\eqref{ass:K-compact_supp}-\eqref{ass:K-monotone}, $K$ is globally non-increasing, thus the integral above is nonnegative for all $k>0$. Consequently, the sum in~\eqref{eq:sum-decomposition} is non-positive and, owing to Assumption~\ref{ass:K-properties}\eqref{ass:K-monotone}, it has at least one strictly negative element corresponding to $q=0$, whence the claim follows.
\end{proof}

\begin{proof}[Proof of Lemma~\ref{lemma:reg_gamma}]
\begin{enumerate}[(i)]
\item From~\eqref{eq:def-flow-map} and~\eqref{eq:Lip_v} we obtain
\begin{align*}
	\abs{\gamma_t(y)-\gamma_t(x)} &\leq \abs{y-x}+\int_0^t\abs{v[\mu_s](\gamma_s(y))-v[\mu_s](\gamma_s(x))}\,ds \\
 	& \leq\abs{y-x}+\xi^N\int_0^t\abs{\gamma_s(y)-\gamma_s(x)}\,ds,
\end{align*}
whence the thesis follows by invoking Gronwall's inequality.
\item Again by~\eqref{eq:def-flow-map} and~\eqref{eq:Lip_v} we have
\begin{align*}
	\abs{\gamma^\nu_t(x)-\gamma^\mu_t(x)} &\leq \int_0^t\abs{v[\nu_s](\gamma^\nu_s(x))-v[\mu_s](\gamma^\mu_s(x))}\,ds \\
	& \leq\xi^N\left(\int_0^t\abs{\gamma^\nu_s(x)-\gamma^\mu_s(x)}\,ds+\frac{1}{N}\int_0^t W_1(\mu_s,\,\nu_s)\,ds\right),
\end{align*}
whence again Gronwall's inequality yields the result. \qedhere
\end{enumerate}
\end{proof}

\begin{proof}[Proof of Proposition~\ref{prop:cont_dep}]
Here we use the expression~\eqref{eq:W1_dual} of $W_1$. Let $\phi\in\Lipone$, then using the notation introduced in Lemma~\ref{lemma:reg_gamma}\eqref{lemma:gamma_mu-nu} and recalling~\eqref{eq:pushfwd} we have:
\begin{align*}
	\int_{\R^d}\phi(x)\,d(\nu_t-\mu_t)(x) &= \int_{\R^d}\phi(x)\,d(\gamma^\nu_t\#\bar{\nu}-\gamma^\mu_t\#\bar{\mu})(x) \\
	& =\int_{\R^d}\phi(\gamma^\nu_t(x))\,d\bar{\nu}(x)-\int_{\R^d}\phi(\gamma^\mu_t(x))\,d\bar{\nu}(x) \\
	&\qquad +\int_{\R^d}\phi(\gamma^\mu_t(x))\,d\bar{\nu}(x)-\int_{\R^d}\phi(\gamma^\mu_t(x))\,d\bar{\mu}(x) \\
	& =\int_{\R^d}(\phi(\gamma^\nu_t(x))-\phi(\gamma^\mu_t(x)))\,d\bar{\nu}(x) \\
	&\qquad +\int_{\R^d}\phi(\gamma^\mu_t(x))\,d(\bar{\nu}-\bar{\mu})(x) \\
	& \leq\int_{\R^d}\abs{\gamma^\nu_t(x)-\gamma^\mu_t(x)}\,d\bar{\nu}(x)+e^{\xi^Nt}W_1(\bar{\mu},\,\bar{\nu}),
\intertext{where in the last term at the right-hand side we have used the fact that the function $x\mapsto\phi(\gamma^\mu_t(x))$ is Lipschitz continuous in view of Lemma~\ref{lemma:reg_gamma}\eqref{lemma:gamma_Lip}. Invoking furthermore Lemma~\ref{lemma:reg_gamma}\eqref{lemma:gamma_mu-nu} we continue as}
	& \leq\frac{\xi^Ne^{\xi^Nt}}{N}\int_{\R^d}\int_0^t W_1(\mu_s,\,\nu_s)\,ds\,d\bar{\nu}+e^{\xi^Nt}W_1(\bar{\mu},\,\bar{\nu}) \\
	& =\xi^Ne^{\xi^Nt}\int_0^t W_1(\mu_s,\,\nu_s)\,ds+e^{\xi^Nt}W_1(\bar{\mu},\,\bar{\nu}).
\end{align*}
Taking the supremum over $\phi$ of both sides we obtain
\[ W_1(\mu_t,\,\nu_t)\leq\xi^Ne^{\xi^NT}\int_0^t W_1(\mu_s,\,\nu_s)\,ds+e^{\xi^Nt}W_1(\bar{\mu},\,\bar{\nu}), \]
where in the first term at the right-hand side we have further used $t\leq T$. Finally we apply Gronwall's inequality and we are done.
\end{proof}

\bibliographystyle{amsplain} 
\bibliography{CaTa-discrcont}  

\providecommand{\bysame}{\leavevmode\hbox to3em{\hrulefill}\thinspace}
\providecommand{\MR}{\relax\ifhmode\unskip\space\fi MR }
\providecommand{\MRhref}[2]{%
  \href{http://www.ams.org/mathscinet-getitem?mr=#1}{#2}
}
\providecommand{\href}[2]{#2}
\begin{thebibliography}{10}

\bibitem{abdellaoui1998JCAM}
T.~Abdellaoui, \emph{Optimal solution of a {M}onge-{K}antorovitch
  transportation problem}, J. Comput. Appl. Math. \textbf{96} (1998), no.~2,
  149--161.

\bibitem{agnelli2015M3AS}
J.~P. Agnelli, F.~Colasuonno, and D.~Knopoff, \emph{A kinetic theory approach
  to the dynamics of crowd evacuation from bounded domains}, Math. Models
  Methods Appl. Sci. \textbf{25} (2015), no.~1, 109--129.

\bibitem{ambrosio2008BOOK}
L.~Ambrosio, N.~Gigli, and G.~Savar{\'e}, \emph{Gradient flows in metric spaces
  and in the space of probability measures}, Lectures in Mathematics ETH
  Z\"urich, Birkh\"auser Verlag, Basel, 2008.

\bibitem{bando1995PRE}
M.~Bando, K.~Hasebe, A.~Nakayama, A.~Shibata, and Y.~Sugiyama, \emph{Dynamical
  model of traffic congestion and numerical simulation}, Phys. Rev. E
  \textbf{51} (1995), no.~2, 1035--1042.

\bibitem{bruno2013PREPRINT}
L.~Bruno, A.~Corbetta, and A.~Tosin, \emph{From individual behaviors to an
  evaluation of the collective evolution of crowds along footbridges},
  Preprint: arXiv:1212.3711, 2013.

\bibitem{bruno2011AMM}
L.~Bruno, A.~Tosin, P.~Tricerri, and F.~Venuti, \emph{Non-local first-order
  modelling of crowd dynamics: {A} multidimensional framework with
  applications}, Appl. Math. Model. \textbf{35} (2011), no.~1, 426--445.

\bibitem{bruno2009JSV}
L.~Bruno and F.~Venuti, \emph{Crowd-structure interaction in footbridges:
  {M}odelling, application to a real case-study and sensitivity analyses}, J.
  Sound Vib. \textbf{323} (2009), no.~1-2, 475--493.

\bibitem{burstedde2001PA}
C.~Burstedde, K.~Klauck, A.~Schadschneider, and J.~Zittartz, \emph{Simulation
  of pedestrian dynamics using a two-dimensional cellular automaton}, Physica A
  \textbf{295} (2001), no.~3-4, 507--525.

\bibitem{canizo2011M3AS}
J.~A. Ca\~{n}izo, J.~A. Carrillo, and J.~Rosado, \emph{A well-posedness theory
  in measures for some kinetic models of collective motion}, Math. Models
  Methods Appl. Sci. \textbf{21} (2011), no.~3, 515--539.

\bibitem{capasso2009SAA}
V.~Capasso and D.~Morale, \emph{Asymptotic behavior of a system of stochastic
  particles subject to nonlocal interactions}, Stoch. Anal. Appl. \textbf{27}
  (2009), no.~3, 574--603.

\bibitem{carrillo2014CISM}
J.~A. Carrillo, Y.-P. Choi, and M.~Hauray, \emph{The derivation of swarming
  models: mean-field limit and {W}asserstein distances}, Collective Dynamics
  from Bacteria to Crowds (A.~Muntean and F.~Toschi, eds.), {CISM}
  {I}nternational {C}entre for {M}echanical {S}ciences, vol. 553, Springer,
  Vienna, 2014, pp.~1--46.

\bibitem{carrillo2009KRM}
J.~A. Carrillo, M.~R. D'Orsogna, and V.~Panferov, \emph{Double milling in
  self-propelled swarms from kinetic theory}, Kinet. Relat. Models \textbf{2}
  (2009), no.~2, 363--378.

\bibitem{carrillo2010SIMA}
J.~A. Carrillo, M.~Fornasier, J.~Rosado, and G.~Toscani, \emph{Asymptotic
  flocking dynamics for the kinetic {C}ucker-{S}male model}, SIAM J. Math.
  Anal. \textbf{42} (2010), no.~1, 218--236.

\bibitem{carrillo2010MSSET}
J.~A. Carrillo, M.~Fornasier, G.~Toscani, and F.~Vecil, \emph{Particle,
  kinetic, and hydrodynamic models of swarming}, Mathematical Modeling of
  Collective Behavior in Socio-Economic and Life Sciences (G.~Naldi,
  L.~Pareschi, and G.~Toscani, eds.), Modeling and Simulation in Science,
  Engineering and Technology, Birkh\"{a}user, Boston, 2010, pp.~297--336.

\bibitem{carrillo2013PD}
J.~A. Carrillo, S.~Martin, and V.~Panferov, \emph{A new interaction potential
  for swarming models}, Phys. D \textbf{260} (2013), 112--126.

\bibitem{colombo2012M3AS}
R.~M. Colombo, M.~Garavello, and M.~L\'{e}cureux-Mercier, \emph{A class of
  nonlocal models for pedestrian traffic}, Math. Models Methods Appl. Sci.
  \textbf{22} (2012), no.~4, 1150023 (34 pages).

\bibitem{colombo2009NARWA}
R.~M. Colombo and M.~D. Rosini, \emph{Existence of nonclassical solutions in a
  pedestrian flow model}, Nonlinear Anal. Real World Appl. \textbf{10} (2009),
  no.~5, 2716--2728.

\bibitem{corbetta2014TRP}
A.~Corbetta, L.~Bruno, A.~Muntean, and F.~Toschi, \emph{High statistics
  measurements of pedestrian dynamics}, Transportation Research Procedia
  \textbf{2} (2014), 96--104.

\bibitem{corbetta2015MBE}
A.~Corbetta, A.~Muntean, and K.~Vafayi, \emph{Parameter estimation of social
  forces in pedestrian dynamics models via a probabilistic method}, Math.
  Biosci. Eng. \textbf{12} (2015), no.~2, 337--356.

\bibitem{cristiani2011MMS}
E.~Cristiani, B.~Piccoli, and A.~Tosin, \emph{Multiscale modeling of granular
  flows with application to crowd dynamics}, Multiscale Model. Simul.
  \textbf{9} (2011), no.~1, 155--182.

\bibitem{cristiani2014BOOK}
\bysame, \emph{Multiscale {M}odeling of {P}edestrian {D}ynamics}, MS\&A:
  Modeling, Simulation and Applications, vol.~12, Springer International
  Publishing, 2014.

\bibitem{daamen2004PhD}
W.~Daamen, \emph{Modelling {P}assenger {F}lows in {P}ublic {T}ransport
  {F}acilities}, Ph.D. thesis, Delft University of Technology, 2004.

\bibitem{degond2013JSP}
P.~Degond, C.~Appert-Rolland, M.~Moussa\"{i}d, J.~Pettr\'{e}, and G.~Theraulaz,
  \emph{A hierarchy of heuristic-based models of crowd dynamics}, J. Stat.
  Phys. \textbf{152} (2013), no.~6, 1033--1068.

\bibitem{degond2013NHM}
P.~Degond, C.~Appert-Rolland, J.~Pettr\'{e}, and G.~Theraulaz,
  \emph{Vision-based macroscopic pedestrian models}, Netw. Heterog. Media
  \textbf{6} (2013), no.~4, 809--839.

\bibitem{dobrushin1979FAA}
R.~L. Dobrushin, \emph{Vlasov equations}, Funct. Anal. App. \textbf{13} (1979),
  no.~2, 115--123.

\bibitem{dorsogna2006PRL}
M.~R. D'Orsogna, Y.~L. Chuang, A.~L. Bertozzi, and L.~S. Chayes,
  \emph{Self-propelled particles with soft-core interactions: ppattern,
  stability, and collapse}, Phys. Rev. Lett. \textbf{96} (2006), no.~10,
  104302/1--4.

\bibitem{duives2013TRC}
D.~C. Duives, W.~Daamen, and S.~P. Hoogendoorn, \emph{State-of-the-art crowd
  motion simulation models}, Transport. Res. C-Emer. \textbf{37} (2013),
  193--209.

\bibitem{evers2014PREPRINT}
J.~Evers, R.~Fetecau, and L.~Ryzhik, \emph{Anisotropic interactions in a
  first-order aggregation model: a proof of concept}, Preprint:
  arXiv:1406.0967, 2014.

\bibitem{fehrenbach2014PREPRINT}
J.~Fehrenbach, J.~Narski, J.~Hua, S.~Lemercier, A.~Jeli\'{c},
  C.~Appert-Rolland, S.~Donikian, J.~Pettr\'{e}, and P.~Degond,
  \emph{Time-delayed {F}ollow-the-{L}eader model for pedestrians walking in
  line}, Preprint: arXiv:1412.7537.

\bibitem{fruin1971BOOK}
J.~J. Fruin, \emph{Pedestrian {P}lanning and {D}esign}, Metropolitan
  {A}ssociation of {U}rban {D}esigners and {E}nvironmental {P}lanners, 1971.

\bibitem{golse2003JEDP}
F.~Golse, \emph{The mean-field limit for the dynamics of large particle
  systems}, Journ\'{e}es \emph{\'{E}quations aux deriv\'{e}es partielles},
  2003.

\bibitem{ha2009CMS}
S.-Y. Ha and J.-G. Liu, \emph{A simple proof of the {C}ucker-{S}male flocking
  dynamics and mean-field limit}, Commun. Math. Sci. \textbf{7} (2009), no.~2,
  297--325.

\bibitem{helbing1995PRE}
D.~Helbing and P.~Moln\'ar, \emph{Social force model for pedestrian dynamics},
  Phys. Rev. E \textbf{51} (1995), no.~5, 4282--4286.

\bibitem{hoogendoorn2003OCAM}
S.~P. Hoogendoorn and P.~H.~L. Bovy, \emph{Simulation of pedestrian flows by
  optimal control and differential games}, Optimal Control Appl. Methods
  \textbf{24} (2003), 153--172.

\bibitem{hughes2002TRB}
R.~L. Hughes, \emph{A continuum theory for the flow of pedestrians},
  Transportation Res. B \textbf{36} (2002), no.~6, 507--535.

\bibitem{johansson2007ACS}
A.~Johansson, D.~Helbing, and P.~K. Shukla, \emph{Specification of the social
  force pedestrian model by evolutionary adjustment to video tracking data},
  Adv. Complex Syst. \textbf{10} (2007), no.~supp02, 271--288.

\bibitem{kirchner2002PA}
A.~Kirchner and A.~Schadschneider, \emph{Simulation of evacuation processes
  using a bionics-inspired cellular automaton model for pedestrian dynamics},
  Physica A \textbf{312} (2002), no.~1-2, 260--276.

\bibitem{maury2007ESAIM}
B.~Maury and J.~Venel, \emph{Un mod\`{e}le de mouvements de foule}, ESAIM:
  Proc. \textbf{18} (2007), 143--152.

\bibitem{piccoli2013AAM}
B.~Piccoli and F.~Rossi, \emph{Transport equation with nonlocal velocity in
  {W}asserstein spaces: convergence of numerical schemes}, Acta Appl. Math.
  \textbf{124} (2013), no.~1, 73--105.

\bibitem{piccoli2011ARMA}
B.~Piccoli and A.~Tosin, \emph{Time-evolving measures and macroscopic modeling
  of pedestrian flow}, Arch. Ration. Mech. Anal. \textbf{199} (2011), no.~3,
  707--738.

\bibitem{polus1983JTE}
A.~Polus, J.~Schofer, and A.~Ushpiz, \emph{Pedestrian flow and level of
  service}, J. Transp. Eng. \textbf{109} (1983), no.~1, 46--56.

\bibitem{seibold2015IPAM}
B.~Seibold, \emph{Resurrection of the {P}ayne-{W}hitham pressure?}, Lecture
  given at the workshop ``Mathematical Foundations of Traffic'', IPAM-UCLA, Los
  Angeles CA, USA. Reference:
  \texttt{http://helper.ipam.ucla.edu/wowzavideo.aspx?vfn=12623.mp4\&vfd=TRA2015},
  September 2015.

\bibitem{seidel2009SIADS}
T.~Seidel, I.~Gasser, and B.~Werner, \emph{Microscopic car-following models
  revisited: {F}rom road works to fundamental diagrams}, SIAM J. Appl. Dyn.
  Syst. \textbf{8} (2009), no.~3, 1305--1323.

\bibitem{seyfried2010CA}
A.~Seyfried, A.~Portz, and A.~Schadschneider, \emph{Phase coexistence in
  congested states of pedestrian dynamics}, Cellular Automata (S.~Bandini,
  S.~Manzoni, H.~Umeo, and G.~Vizzari, eds.), Lecture Notes in Computer
  Science, vol. 6350, Springer, Berlin Heidelberg, 2010, pp.~496--505.

\bibitem{seyfried2008CA}
A.~Seyfried and A.~Schadschneider, \emph{Fundamental diagram and validation of
  crowd models}, Cellular Automata (H.~Umeo, S.~Morishita, K.~Nishinari,
  T.~Komatsuzaki, and S.~Bandini, eds.), Lecture Notes in Computer Science,
  vol. 5191, Springer, Berlin Heidelberg, 2008, pp.~563--566 (English).

\bibitem{tosin2011NHM}
A.~Tosin and P.~Frasca, \emph{Existence and approximation of probability
  measure solutions to models of collective behaviors}, Netw. Heterog. Media
  \textbf{6} (2011), no.~3, 561--596.

\bibitem{twarogowska2014AMM}
M.~Twarogowska, P.~Goatin, and R.~Duvigneau, \emph{Macroscopic modelling and
  simulations of room evacuation}, Appl. Math. Model. \textbf{38} (2014),
  no.~24, 5781--5795.

\bibitem{venuti2009PLR}
F.~Venuti and L.~Bruno, \emph{Crowd-structure interaction in lively footbridges
  under synchronous lateral excitation: A literature review}, Phys. Life Rev.
  \textbf{6} (2009), no.~3, 176--206.

\bibitem{villani2009BOOK}
C.~Villani, \emph{Optimal transport -- {O}ld and new}, Grundlehren der
  Mathematischen Wissenschaften [Fundamental Principles of Mathematical
  Sciences], Springer-Verlag, Berlin, 2009.

\bibitem{zanlungo2012PONE}
F.~Zanlungo, T.~Ikeda, and T.~Kanda, \emph{A microscopic ``social norm'' model
  to obtain realistic macroscopic velocity and density pedestrian
  distributions}, PLoS ONE \textbf{7} (2012), no.~12, e50720.

\bibitem{zhang2014PLA}
J.~Zhang, W.~Mehner, S.~Holl, M.~Boltes, E.~Andresen, A.~Schadschneider, and
  A.~Seyfried, \emph{Universal flow-density relation of single-file bicycle,
  pedestrian and car motion}, Phys. Lett. A \textbf{378} (2014), no.~44,
  3274--3277.

\end{thebibliography}
\end{document}